\RequirePackage{fix-cm}
\documentclass[epjc3, envcountsect]{svjour3}  
\smartqed  % flush right qed marks, e.g. at end of proof
\RequirePackage{graphicx}
\RequirePackage{mathptmx}      % use Times fonts if available on your TeX system
\usepackage{amsmath}
\usepackage{amsfonts}
\usepackage{cleveref}
\usepackage{nicefrac}
\usepackage{enumerate}
\usepackage[labelfont=bf,labelsep=space]{caption}
\usepackage{tikz}
\usetikzlibrary{patterns}
\usetikzlibrary{patterns.meta}

\def\makeheadbox{\relax}

\newcommand{\VECVECSIX}[6]{\left[\begin{array}{c}
			#1 \\ #2 \\#3 \\ #4 \\ #5 \\ #6 \end{array}\right]} % 6x1 Matrix

\renewcommand{\epsilon}{\varepsilon}
\renewcommand{\phi}{\varphi}

\begin{document}
\title{Remarks on Exponential Stability for a Coupled System of Elasticity and Thermoelasticity with Second Sound%\thanksref{t1}
}

\titlerunning{Exponential Stability for a Coupled System of Elasticity and Thermoelasticity}        % if too long for running head

\author{Manuel Rissel\thanksref{e1,addr1}
        \and
        Ya-Guang Wang\thanksref{e2,addr2} %etc.
}

\thankstext{e1}{e-mail: manuel.rissel@sjtu.edu.cn}
\thankstext{e2}{e-mail: ygwang@sjtu.edu.cn}

\institute{School of Mathematical Sciences,	Shanghai Jiao Tong University, Shanghai, 200240, P. R. China \label{addr1}
	\and
    School of Mathematical Sciences, Center of Applied Mathematics, MOE-LSC and SHL-MAC, Shanghai Jiao Tong University, Shanghai, 200240, P. R. China \label{addr2}
           %\and
           %\emph{Present Address:} if needed\label{addr3}
}

\date{Received: date / Accepted: date}

\maketitle

\begin{abstract}
We study the large time behavior of solutions to a linear transmission problem in one space dimension. The problem at hand models a thermoelastic material with second sound confined by a purely elastic one. We shall characterize all equilibrium states of the considered system and prove that every solution approaches one designated equilibrium state with an exponential rate as time goes to infinity. Hereto, we apply methods from the theory of strongly continuous semigroups. In particular, we obtain uniform resolvent bounds for the underlying generator. This removes the largeness assumption of elastic wave speeds imposed in [Y.P. Meng and Y.G. Wang,  Anal. Appl. (Singap.) 13 (2015)] for having an exponential energy decay rate when the problem only has the trivial equilibrium. In an appendix we provide a similar exponential stability result for the case where heat conduction is modeled using Fourier's law.
\keywords{transmission problem \and thermoelasticity \and elasticity \and second sound \and exponential stability}
 \subclass{35B35,  35B40, 35M33, 47D06 }
\end{abstract}

\section{Introduction}

Many interesting applications give rise to transmission problems, which are systems of differential equations with discontinuous coefficients and transmission conditions imposed on some interfaces. In this note we are concerned with a transmission problem for a coupled system of elasticity and thermoelasticity, as illustrated in \Cref{Figure:IllustrationOfBar}. In this work, the heat conduction in thermoelasticity obeys the Cattaneo law, which transforms the classical thermoelastic system of hyperbolic-parabolic type into the thermoelastic system with second sound, a strictly hyperbolic one, cf. \cite{RackeHandbook}.  We aim to develop the  linear semigroup  theory to re-study the long time stability analysis of this problem initiated in \cite{WangMeng} via the Lyapunov argument. In particular, by the means of uniform resolvent bounds, we shall show that all equilibria are exponentially stable. While our motivation stems from the work \cite{WangMeng}, where only the model with Cattaneo's law is considered, it seems to us that the same setting but with classical Fourier's law for heat conduction has not been studied, in particular in the semigroup context, as well. For this reason, in the appendix, we also provide details on how to obtain exponential stability when Fourier's law is employed and the thermoelastic part of the bar is modeled in the classical way. The one dimensional systems investigated here serve as toy-models for the curl-free part of considerably more complicated vectorial cases and are also of interest from the perspective of control theory.

\begin{figure}[ht!]
	\centering 
	\captionsetup{justification=centering}
		\begin{tikzpicture}	
		\draw[thick,->] (0,0) -- (6.5,0);
		\draw (0,0.1) -- (0, -0.1) node[anchor=north] {$0$};
		\draw (2,0.1) -- (2, -0.1) node[anchor=north] {$L_1$};	
		\draw (5,0.1) -- (5, -0.1) node[anchor=north] {$L_2$};
		\draw (6,0.1) -- (6, -0.1) node[anchor=north] {$L_3$};
		\draw[pattern=north east lines,pattern color=black]   (0,0.2) rectangle (2,0.8);	
		\draw[pattern=dots,pattern color=black]   (2,0.2) rectangle (5,0.8);
		\draw[pattern=north east lines,pattern color=black]   (5,0.2) rectangle (6,0.8);
		\draw [fill=white] (0.75,0.3) rectangle (1.25,0.7) node[pos=.5] {E};
		\draw [fill=white] (3.25,0.3) rectangle (3.75,0.7) node[pos=.5] {TE};
		\draw [fill=white] (5.25,0.3) rectangle (5.75,0.7) node[pos=.5] {E};
		\end{tikzpicture}
		%	\vspace*{8pt}
		\caption{Illustration of an elastic(E)-theormoelastic(TE)-elastic bar}
		\label{Figure:IllustrationOfBar}
\end{figure}
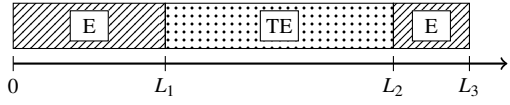

The dynamics of the elastic-thermoelastic-elastic bar under consideration, with Cattaneo's law for heat conduction, are modeled by the following set of coupled partial differential equations,
\begin{equation}\label{equation:TimeDomainSystem:eq1}
\begin{aligned}
& u_{tt} - a u_{xx} + m \theta_x  & = && 0 && \mbox{ in } && [L_1, L_2]\times\mathbb{R}_+, \\
& \theta_t + k q_x + mu_{xt} & = && 0 && \mbox{ in } && [L_1, L_2]\times\mathbb{R}_+, \\
& \tau q_t + q + k\theta_x & = && 0 && \mbox{ in } && [L_1, L_2]\times\mathbb{R}_+, \\
& v_{tt} - b v_{xx} & = && 0 && \mbox{ in } && [0, L_1] \cup [L_2, L_3] \times\mathbb{R}_+,
\end{aligned}
\end{equation}
with transmission conditions at the interfaces
\begin{equation}\label{equation:TimeDomainSystem:eq2}
\begin{aligned}
& u(L_i, \cdot) && = && v(L_i, \cdot) &&  \, && \,  && \mbox{ in } && \mathbb{R}_+\,\, && \quad (i=1,2),\\
& au_x(L_i, \cdot) && - && m \theta(L_i, \cdot) && = && bv_x(L_i, \cdot) && \mbox{ in } && \mathbb{R}_+\,\,  && \quad (i=1,2),\\
\end{aligned}
\end{equation}
and usual Dirichlet conditions for the displacement at the endpoints of the bar as well as for the heat flux at the insulated ends of the middle part,
\begin{equation}\label{equation:TimeDomainSystem:eq3}
\begin{aligned}
& q(L_1, \cdot) && = && q(L_2, \cdot) && = && 0 && \mbox{ in } && \mathbb{R}_+, && \, \\
& v(0, \cdot) && = && v(L_3, \cdot) && = && 0 && \mbox{ in } && \mathbb{R}_+, &&
\end{aligned}
\end{equation}
further complemented by initial conditions
\begin{equation}\label{equation:TimeDomainSystem:eq4}
\begin{aligned}
& u(\cdot,0)  =  u_0, ~u_t(\cdot,0) = u_1, ~\theta(\cdot, 0) = \theta_0,~ q(\cdot, 0) = q_0 && \mbox{ in } && [L_1, L_2],\\
& v(\cdot, 0)  =  v_0,~ v_t(\cdot, 0) = v_1 && \mbox{ in } &&  [0, L_1] \cup [L_2, L_3].
\end{aligned}
\end{equation}

The unknown functions to be determined are the displacement $u$, the temperature $\theta$ and the heat flux $q$ for the middle part of the bar, while $v$ is the unknown displacement of the outer parts. Moreover, $a, m, k, \tau, b > 0$ are strictly positive parameters, which are arbitrary but fixed and represent physical properties of the considered materials. The limit case $\tau = 0$ corresponds to Fourier's law of heat conduction, where $q = -k\theta_x$, regarding which we refer to \cite{RackeHandbook}, some references below and the appendix. As in  \cite{WangMeng}, we assume that the elastic speed $\sqrt{b} > 0$ is identical for the material in $[0, L_1]$ and $[L_2, L_3]$. The problem \eqref{equation:TimeDomainSystem:eq1}-\eqref{equation:TimeDomainSystem:eq4} was studied in \cite{WangMeng} by using the energy method, in which they obtained that if the initial data satisfy the constraint
\begin{equation}
\begin{aligned}\label{equation:ConditionStationarySol}
\int_{L_1}^{L_2}  \theta_0 \, dx  + m (u_0(L_2) - u_0(L_1)) = 0,
\end{aligned}
\end{equation}
and the wave speed $\sqrt{b}$ is large enough in the sense of \cite[(3.37) of Theorem 3.8]{WangMeng}, then the solution of the above problem decays exponentially to zero in the energy space when $t$ goes to infinity.

The aim of this article is to study the problem \eqref{equation:TimeDomainSystem:eq1}-\eqref{equation:TimeDomainSystem:eq4} by using a semigroup approach, and we shall deduce that the localized dissipation of the thermoelasticity in the equations 
\eqref{equation:TimeDomainSystem:eq1}-\eqref{equation:TimeDomainSystem:eq4} is strong enough to
guarantee any stationary state of the problem \eqref{equation:TimeDomainSystem:eq1}-\eqref{equation:TimeDomainSystem:eq4} being exponentially stable, and the constraint  \eqref{equation:ConditionStationarySol} being equivalent to that this problem has only the trivial equilibrium. 

Thus, in contrast to \cite{WangMeng}, there are two novelties at least in this paper, one is that by using the semigroup approach, without the largeness requirement of the wave speed $\sqrt{b}$ imposed in 
\cite{WangMeng}, we obtain that the solution of \eqref{equation:TimeDomainSystem:eq1}-\eqref{equation:TimeDomainSystem:eq4} decays to zero exponentially when the time $t$ goes to infinity under the assumption \eqref{equation:ConditionStationarySol} of the initial data, the second is that any solution of this problem goes to a steady state exponentially when $t$ goes to infinity even without the constraint \eqref{equation:ConditionStationarySol}. Besides that, it seems that exponential stability for solutions to \eqref{equation:TimeDomainSystem:eq1}-\eqref{equation:TimeDomainSystem:eq4} with $\tau = 0$ (using Fourier's law) is likewise not available in the literature up to now, in particular from a semigroup point of view. Therefore, the discussion we provide in the appendix might be of interest as well. For cases where the thermoelastic part is located in at least one end of the bar, exponential stability has been obtained in a semilinear context with Cattaneo's law by Sare, Mu\~{n}oz Rivera \& Racke \cite{RackeSareRivera} and with Fourier's law by Marzocchi, Mu\~{n}oz Rivera \&  Naso \cite{Rivera1D_threeparts}. These works both employ an energy method for treating the semilinear problem directly.

The essential difference of our approach, compared to the energy functional method used in \cite{WangMeng}, is that we work in the frequency domain and study the generator of the semigroup which governs the linear dynamics. While an energy functional approach, employed for example in \cite{Alves,Fatori,LiuChen,WangMeng,RiveraPortillo,RackeThermoelasticitySecondSound,RackeSareRivera,Bravo}, is beneficial in that it directly applies to some semilinear settings, there might appear disadvantages when necessary conditions for stability are desired or in cases where suitable Lyapunov type functionals are difficult to obtain. On the contrary, the linear semigroup approach offers several equivalent characterizations for exponential stability and some criteria for weaker forms such as polynomial or strong stability, which have also proven to be helpful for showing the lack of exponential stability, see for example \cite{Alves,Alves2,batty1,batty2,HanXu,rackeVisco,Ng_Seifert,SchnaubeltRzepnicki,waterwaves,ZhangZuazua,zuazua}. 

Regarding \eqref{equation:TimeDomainSystem:eq1}-\eqref{equation:TimeDomainSystem:eq4}, we shall observe below, see \eqref{equation:adjointFormula0}, that the generator $\mathcal{A}$ of the underlying semigroup allows a natural decomposition $\mathcal{A} = \mathcal{C} + \mathcal{E}$ with $\mathcal{C}$ skew-adjoint and $\mathcal{E}$ bounded selfadjoint. We note that therefore the dissipative effects, entering this system  only through the heatflux $q$, are rather weak. In particular, the spectrum of $\mathcal{A}$ lives in an infinite strip of thickness $\tau^{-1}$ touching the imaginary axis, see \Cref{remark:LocationOfSpectrum} and \Cref{proposition:StationarySolutions} below. This is due to the use of Cattaneo's law for modeling the thermoelastic component as a hyperbolic system in order to rule out the paradox of infinite speeds of propagation, observed for classical heat diffusion models following Fourier's law \cite{RackeHandbook}. With Fourier's law, the perturbation $\mathcal{E}$ becomes an unbounded one, introducing some parabolic features. In both cases, the nice structure of the generator relies on the natural choice of the transmission conditions \eqref{equation:TimeDomainSystem:eq2}, which guarantee a matching of the forces at the interfaces. 

For a solution $U = [u(t,x), v(t,x), \theta(t,x), q(t,x), u_t(t,x), v_t(t,x)]'$ to \eqref{equation:TimeDomainSystem:eq1} - \eqref{equation:TimeDomainSystem:eq4} the energy $U$ at time $t$ is defined by
\begin{equation}\label{equation:Energy}
\begin{aligned}
E_U(t) &:= && \frac{1}{2} \left[\int_0^{L_1} + \int_{L_2}^{L_3}\right] \left(b |v_x(t,x)|^2  + |v_t(t,x)|^2\right) \, dx \\
& && + \frac{1}{2}\int_{L_1}^{L_2}  \left( a |u_x(t,x)|^2 + |u_t(t,x)|^2 + |\theta(t,x)|^2 + \tau |q(t,x)|^2 \right) \, dx.
\end{aligned}
\end{equation}

Our main result, stated later in \Cref{theorem:MainResult}, can for the time being be formulated as below in \Cref{theorem_temp}. For the case of Fourier's law we refer to \Cref{fourier:theorem:MainResult}. 
\begin{theorem}\label{theorem_temp}
	There exist constants $M, \omega > 0$ such that, for every initial datum
	\[
		U_0 = [u_0,v_0, q_0, \theta_0, u_1, v_1]'
	\]
	with finite energy $E_{U_0}$, the corresponding finite energy solution 
	\[
		U = U(t,x) = [u(t,x), v(t,x), q(t,x), \theta(t,x)]'
	\]
	to \eqref{equation:TimeDomainSystem:eq1} - \eqref{equation:TimeDomainSystem:eq4} satsifies
	\[
	E_{U - W_0}(t) \leq  M \operatorname{e}^{-\omega t} \quad (t \geq 0),
	\]
	where $W_0$ is a stationary solution to \eqref{equation:TimeDomainSystem:eq1} - \eqref{equation:TimeDomainSystem:eq4} and uniquely determined by $U_0$.	
\end{theorem}

The structure of the article at hand is as follows. In \Cref{Section:Wellposedness}, the well-posedness of \eqref{equation:TimeDomainSystem:eq1}-\eqref{equation:TimeDomainSystem:eq4} is settled in the framework of strongly continuous semigroups, some spectral properties of the generator are obtained and all stationary states are characterized. Then, in \Cref{Section:exponentialStability}, the uniform exponential stability result of \Cref{theorem_temp} is established. Some calculations are outsourced to \Cref{proofs}. Finally, some details regarding the classical model with Fourier's law, culminating in \Cref{fourier:theorem:MainResult}, are presented in form of an appendix.

\section{Well-posedness and stationary solutions}\label{Section:Wellposedness}
Given a linear operator $T$, its spectrum is denoted by $\sigma(T)$, the point spectrum of $T$ by $\sigma_p(T)$ and the resolvent set of $T$ by $\rho(T)$. For a complex number $\lambda$, the notation $\lambda - T$ is short for $\lambda I - T$. Furthermore $\Re \lambda$ stands for the real part of $\lambda \in \mathbb{C}$ and $\overline{\lambda}$ for its complex conjugate. We define the state space for solutions of \eqref{equation:TimeDomainSystem:eq1} - \eqref{equation:TimeDomainSystem:eq4} by
\begin{equation*}
\begin{aligned}
\mathcal{H}  & := &&   \Big\{ [u^1, v^1, \theta , q, u^2, v^2]'  \in  H^1(L_1, L_2) \times H^1((0, L_1) \cup (L_2, L_3)) \\
& &&\times L^2(L_1, L_2) \times L^2(L_1, L_2)  \times L^2(L_1, L_2) \times L^2((0, L_1) \cup (L_2, L_3)) \\
& &&\mbox{ such that } v^1(0) = v^1(L_3) = 0 \mbox{ and }  u^1(L_i) = v^1(L_i) \mbox{ for } i= 1,2   \Big\},
\end{aligned}
\end{equation*}
and introduce on $\mathcal{H}$ the inner product
\begin{equation*}
\begin{aligned}
\langle  U, \tilde{U} \rangle_{\mathcal{H}} & := && \int_{L_1}^{L_2} \left( a u^1_x \overline{\tilde{u^1_x}} + u^2 \overline{\tilde{u^2}} + \theta \overline{\tilde{\theta}} +  \tau q \overline{\tilde{q}}\right)  \, dx + \left[\int_0^{L_1} + \int_{L_2}^{L_3}\right] \left( b v^1_x \overline{\tilde{v^1}_x}  + v^2 \overline{\tilde{v^2}} \right) \, dx,
\end{aligned}
\end{equation*}
which corresponds to the energy functional $E_U(t)$ given in \eqref{equation:Energy}. Let $U := [u^1, v^1, \theta , q, u^2, v^2]'  \in \mathcal{H}$ be given and define on $[0, L_3]$ the function $w(x) := u^1(x)$ if $x \in [L_1, L_2]$, $w(x) := 	v^1(x)$ if $x \in [0, L_1] \cup [L_2, L_3]$, which is continuous by definition of $\mathcal{H}$ and belongs to $H^1_0(0, L_3)$. Then, by applying the Poincar\'e inequality to $w$, there exists a constant $C > 0$ such that
\begin{equation}\label{equation:Poincare}
\begin{aligned}
	\|u^1\|_{L^2(L_1, L_2)}^2 + \|v^1\|_{L^2((0, L_1) \cup (L_2, L_3))}^2 & = && \int_0^{L_3} |w|^2 \, dx \\
	& \leq && C \int_0^{L_3} |w_x|^2 \, dx \\
	& = && C \left( \|u^1_x\|_{L^2(L_1, L_2)}^2 + \|v^1_x\|_{L^2((0, L_1) \cup (L_2, L_3))}^2 \right),
\end{aligned}
\end{equation}
in particular, $\mathcal{H}$ together with the norm $\|\cdot\|_{\mathcal{H}}$ induced by $\langle \cdot, \cdot \rangle_{\mathcal{H}}$ is a Hilbert space.

We define the linear operator $\mathcal{A} \colon \mathcal{H} \supseteq D(\mathcal{A}) \to \mathcal{H}$ for $U =  [u^1, v^1, \theta, q, u^2, v^2]'$ via  
\[
\mathcal{A}U := \left[\begin{array}{cccccc}
0 & 0 & 0 & 0 & 1 & 0 \\ 
0 & 0 & 0 & 0 & 0 & 1\\
0 & 0 & 0 & -k \partial_x & -m\partial_x & 0 \\ 
0 & 0 & -\frac{k}{\tau}\partial_x & - \frac{1}{\tau} & 0 & 0\\
a\partial_{xx} & 0 & -m\partial_x & 0 & 0 & 0 \\ 
0 & b\partial_{xx} & 0 & 0 & 0 & 0
\end{array}\right] \VECVECSIX{u^1}{v^1}{\theta}{q}{u^2}{v^2} := \VECVECSIX{u^2}{v^2}{-(kq +mu^2)_x}{-\frac{k}{\tau}\theta_x-\frac{1}{\tau}q}{(au^1_{x} -m\theta)_x}{bv^1_{xx}},
\]
on its domain
\begin{equation*}
\begin{aligned}
	D(\mathcal{A}) := \left\{ U  \in \mathcal{H}  \,\, \left| \right. \,\, \mathcal{A}U \in \mathcal{H}, \, au^1_x(L_i) - m\theta(L_i) = bv^1_x(L_i), \, q(L_i) = 0, \, i = 1,2   \right\},
\end{aligned}
\end{equation*}
which can explicitly be described by
\begin{equation*}
\begin{aligned}
D(\mathcal{A}) & = && \Big\{ U  \in \mathcal{H} \cap \big[ H^2(L_1, L_2) \times H^2((0, L_1) \cup (L_2, L_3)) \times H^1(L_1, L_2)  \times H^1(L_1, L_2) \\
& && \times H^1(L_1, L_2) \times H^1((0, L_1) \cup (L_2, L_3)) \big]  \,\, \Big| \,\,  u^2(L_i) = v^2(L_i), \\
& &&  v^2(0) = v^2(L_3) = 0, \, au^1_x(L_i) - m\theta(L_i) = bv^1_x(L_i), \, q(L_i) = 0, \, i = 1,2   \Big\}.
\end{aligned}
\end{equation*}
In order to verify this equality, note that $\mathcal{A}U \in \mathcal{H}$ implies $u^2 \in H^1(L_1, L_2), v^2 \in H^1((0, L_1) \cup (L_2, L_3))$ and $u^2(L_i) = v^2(L_i), i= 1,2$. Moreover, $kq + mu^2 =: z \in H^1(L_1, L_2)$, hence $q = k^{-1}(z-mu^2) \in H^1(L_1, L_2)$. Now, one can read from $\mathcal{A}U \in \mathcal{H}$ that $\theta \in H^1(L_1, L_2)$ and $au^1_x - m \theta =: \tilde{z} \in H^1(L_1, L_2)$, which imply together that $u^1_x = a^{-1} (\tilde{z} - m \theta) \in H^1(L_1, L_2)$, hence $u^1 \in H^2(L_1, L_2)$. Finally, one can directly see that $v^1 \in H^2((0, L_1) \cup (L_2, L_3))$. In particular the pointwise boundary conditions are justified thanks to the Sobolev embedding $H^1(I) \hookrightarrow C(I)$ for every open $I \subseteq \mathbb{R}$. The other inclusion follows directly.

Through $u^2 = u_t$ and $v^2 = v_t$, the original equations \eqref{equation:TimeDomainSystem:eq1} - \eqref{equation:TimeDomainSystem:eq4} can be transformed into the following Cauchy problem in an abstract form
\begin{equation}\label{equation:Cauchyproblem}
\begin{aligned}
& \frac{d}{dt}U(t) & = && & \mathcal{A} U(t), && t \geq 0, \\
& U(0) & = && & U_0 \in \mathcal{H}. &&
\end{aligned}
\end{equation}

We introduce, according to  \cite[Page 145f]{EngelNagel}, two solution concepts for this Cauchy problem.  Consider a $\mathcal{H}$-valued function $U \colon \mathbb{R}_+ \to \mathcal{H}$. If $U \in C^0([0, \infty); D(\mathcal{A})) \cap C^1([0, \infty); \mathcal{H})$ and $U$ satisfies \eqref{equation:Cauchyproblem} in $\mathcal{H}$, then $U$ is said to be a classical solution to \eqref{equation:Cauchyproblem}. If $U \in C^0([0, \infty); \mathcal{H})$ with $\int_0^t U(s) \, ds \in D(\mathcal{A})$ for all $t \geq 0$ and
\[
U(t) = \mathcal{A} \int_0^t U(s) \, ds + U_0,
\]
then $U$ is called a mild solution to \eqref{equation:Cauchyproblem}. Assume that $\mathcal{A}$ generates a strongly continuous semigroup $(T(t))_{t \geq 0}$ in $\mathcal{H}$. Due to the definition of $D(\mathcal{A})$, the existence of a classical solutions is equivalent to $U_0 \in D(\mathcal{A})$. For $U_0 \in D(\mathcal{A})$, the unique classical solution to \eqref{equation:Cauchyproblem} is then given by $U(t) = T(t)U_0$. If, more general, $U_0 \in \mathcal{H}$, then the unique mild solution is given again by $U(t) = T(t)U_0$. The latter statements can be found in \cite[6.2 Proposition, 6.4 Proposition]{EngelNagel}. Regarding the correspondence of solutions to \eqref{equation:Cauchyproblem} and of solutions to \eqref{equation:TimeDomainSystem:eq1} - \eqref{equation:TimeDomainSystem:eq4} see also the discussion in \cite{LiuZheng}.

\begin{proposition}\label{proposition:Semigroup}
	The operator $\mathcal{A}$ generates a strongly continuous contraction semigroup $(T(t))_{t \geq 0}$ on $\mathcal{H}$. 
\end{proposition}
\begin{proof}
	\begin{enumerate}[1)]
		\item We start by showing that $D(\mathcal{A})$ is dense in $\mathcal{H}$ and hereto arbitrarily choose $U := [u^1, v^1, \theta, q, u^2, v^2]' \in \mathcal{H}$. Then it is left to construct a sequence in $D(\mathcal{A})$ which converges to $U$ in $\mathcal{H}$. First, we define on $[0, L_3]$ the function
		\[ w^1(x) := \begin{cases} 
		u^1(x) & x \in [L_1, L_2], \\
		v^1(x) & x \in [0, L_1] \cup [L_2, L_3],
		\end{cases}
		\]
		where $w^1 \in H^1_0(0, L_3)$ by definition of $\mathcal{H}$. Moreover, let $\gamma \in C^{\infty}_0(0,L_3)$ be a smooth function, which satisfies
		\[
		\gamma(x) := \begin{cases} 
		0 & x \in [0, \frac{L_1}{4}) \cup (L_3-\frac{L_3-L_2}{4}, L_3],\\
		u^1(L_1) & x \in (L_1 - \frac{L_1}{4}, L_1 + \frac{L_2-L_1}{4}), \\
		u^1(L_2) & x \in (L_2 - \frac{L_2-L_1}{4}, L_2+\frac{L_3-L_2}{4}).
		\end{cases}
		\]
		In particular it holds $\gamma(L_i) = u^1(L_i)$, $i=1,2$ and $\gamma_x(x) = 0$ for $x \in \{0, L_1, L_2, L_3\}$. Employing the density of $C^{\infty}_0(I)$ in $H^1_0(I)$ for open $I \subseteq \mathbb{R}$, there exist sequences $(\alpha_n)_{n \in \mathbb{N}} \in C^{\infty}_0((0, L_1)\cup(L_2,L_3)), (\beta_n)_{n \in \mathbb{N}} \in C^{\infty}_0(L_1, L_2)$, such that for $n \to +\infty$
		\begin{equation*}
		\begin{aligned}
			\alpha_n & \to && (w^1 - \gamma)_{|_{(0, L_1)\cup(L_2,L_3)}} & \mbox{ in } & H^1((0, L_1)\cup(L_2,L_3)),\\
			\beta_n & \to && (w^1 - \gamma)_{|_{(L_1,L_2)}} & \mbox{ in } & H^1(L_1, L_2).
		\end{aligned}
		\end{equation*}
		Now define for each $n \in \mathbb{N}$,
		\begin{equation*}
		\begin{aligned}
		v^1_n(x) & := && \alpha_n(x) + \gamma(x) \quad (x \in [0, L_1]\cup[L_2,L_3]),\\
		u^1_n(x) & := && \beta_n(x) + \gamma(x) \quad (x \in [L_1, L_2]),
		\end{aligned}
		\end{equation*}
		which satisfy for $n \to +\infty$ that $v^1_n \to v^1$ in $H^1((0, L_1)\cup(L_2,L_3))$ as well as $u^1_n \to u^1$ in $H^1(L_1, L_2)$. Moreover, for each $n \in \mathbb{N}$, ones has $u^1_n(L_i) = v^1_n(L_i)$ and $u^1_x(L_i) = v^1_x(L_i) = 0$, $i=1,2$. Finally, with the help of the dense inclusion $C^{\infty}_0(I) \subseteq L^2(I)$ for open $I \subseteq \mathbb{R}$, one can choose 
		\begin{equation*}
		\begin{aligned}
		(\theta_n)_{n \in \mathbb{N}}, (q_n)_{n \in \mathbb{N}}, (u^2_n)_{n \in \mathbb{N}} & \subseteq && C^{\infty}_0(L_1, L_2),\\
		 (v^2_n)_{n \in \mathbb{N}} & \subseteq && C^{\infty}_0((0, L_1)\cup(L_2,L_3)),
		\end{aligned}
		\end{equation*}
		such that for $n \to +\infty$ one has $\theta_n \to \theta$ in $L^2(L_1, L_2)$, $q_n \to q$ in $L^2(L_1, L_2)$, $u^2_n \to u^2$ in $L^2(L_1, L_2)$ and $v^2_n \to v^2$ in $L^2((0, L_1)\cup(L_2, L_3))$. By construction, it holds that
		\[
			U_n := [u^1_n, v^1_n, \theta_n, q_n, u^2_n, v^2_n]' \in D(\mathcal{A})
		\]
		for all $n \in \mathbb{N}$ and $U_n \to U$ in $\mathcal{H}$ as  $n \to +\infty$. Thus, $D(\mathcal{A})$ is dense in $\mathcal{H}$.
		\item A straight forward calculation, employing the boundary and transmission conditions, reveals that the $\mathcal{H}$-adjoint $\mathcal{A}^*$ of $\mathcal{A}$ satisfies $D(\mathcal{A}^*) = D(\mathcal{A})$ and is given by
		\[
		\mathcal{A}^* U = \left[\begin{array}{cccccc}
		0 & 0 & 0 & 0 & -1 & 0 \\ 
		0 & 0 & 0 & 0 & 0 & -1 \\
		0 & 0 & 0 & k \partial_x & m\partial_x & 0 \\ 
		0 & 0 & \frac{k}{\tau} \partial_x & - \frac{1}{\tau} & 0 & 0\\
		-a\partial_{xx} & 0 & m\partial_x & 0 & 0 & 0 \\ 
		0 & -b\partial_{xx} & 0 & 0 & 0 & 0
		\end{array}\right] \VECVECSIX{u^1}{v^1}{\theta}{q}{u^2}{v^2} = \VECVECSIX{-u^2}{-v^2}{(kq + mu^2)_x }{-\frac{1}{\tau}q + \frac{k}{\tau} \theta_x}{-(au^1_{x} -m\theta)_x}{-bv^1_{xx}}. 	
		\]	
		We observe that
		\begin{equation}\label{equation:adjointFormula0}
		\begin{aligned}
		\mathcal{A} + E_{\tau} = - \mathcal{A}^* - E_{\tau},
		\end{aligned}
		\end{equation}
		where $E_{\tau}$ denotes the linear bounded self-adjoint operator $\mathcal{H} \to \mathcal{H}$, defined by
		\begin{equation}\label{equation:adjointFormula}
		\begin{aligned}
		E_{\tau}\colon [u^1, v^1, \theta, q, u^2, v^2]' \mapsto \tau^{-1} [0, 0, 0, q, 0, 0]'.
		\end{aligned}
		\end{equation}
		In particular, one can verify that $(\mathcal{A}^*)^* = \mathcal{A}$ either by using directly the definition of the adjoint and integration by parts, or by alternatively employing \eqref{equation:adjointFormula0} and the fact that $E_{\tau}$ is bounded, which leads to
		\[
		(\mathcal{A}^*)^* = (- \mathcal{A} - 2 E_{\tau})^* = - \mathcal{A}^* - 2 E_{\tau}^* = \mathcal{A}.
		\]
		Thus, $\mathcal{A}$ is the $\mathcal{H}$-adjoint of a densely defined operator in $\mathcal{H}$ and therefore closed by the standard theory for unbounded operators in Hilbert spaces.
		\item By using the boundary and transmission conditions, one obtains for $U \in D(\mathcal{A})$, 
		\begin{equation}\label{equation:dissipation}
		\begin{aligned}
		\Re \langle \mathcal{A}U, U \rangle_{\mathcal{H}} & = &&  \Re \Bigg\{  \int_{L_1}^{L_2} (a u^2_x \overline{u^1_x}  - k q_x \overline{\theta } - m  u^2_x \overline{\theta} - k  \theta_x \overline{q} -  q \overline{q} + a u^1_{xx} \overline{u^2} - m \theta_x \overline{u^2}) \, dx \\
		& && + \left[ \int_0^{L_1} + \int_{L_2}^{L_3} \right]   \left(b v^2_x \overline{v^1_x} +   b v^1_{xx} \overline{v^2}\right) \, dx  \Bigg\}\\
		& = && - \int_{L_1}^{L_2} q \overline{q} \, dx \\
		& \leq && 0.
		\end{aligned}
		\end{equation}
		In the same way $\mathcal{A}^*$ satisfies
		\[
		\Re \langle \mathcal{A}^*U, U \rangle_{\mathcal{H}}  =  - \int_{L_1}^{L_2} q \overline{q} \, dx  \leq  0
		\]
		for every $U \in D(\mathcal{A}^*) = D(\mathcal{A})$ as well and therefore, because $\mathcal{H}$ is a Hilbert space, $\mathcal{A}$ and $\mathcal{A}^*$ are both dissipative. Hence, by a corollary of the Lumer \& Phillips Theorem, cf. \cite[3.17 Corollary]{EngelNagel}, $\mathcal{A}$ generates a strongly continuous contraction semigroup. \qed
	\end{enumerate}
\end{proof}
Due to being the generator of a strongly continuous contraction semigroup, $\mathcal{A}$ satisfies $\{\lambda \in \mathbb{C} \, | \,  \Re \lambda >0 \} \subseteq \rho(\mathcal{A})$ and thanks to the compact embedding $D(\mathcal{A}) \hookrightarrow \mathcal{H}$, $\sigma(\mathcal{A})$ consists only of eigenvalues. We refer to \cite[1.19 Corollary]{EngelNagel} and \cite[4.25 Proposition]{EngelNagel}.

\begin{remark}\label{remark:LocationOfSpectrum} 
	By using the definition of the inner product in $\mathcal{H}$ one can observe from \eqref{equation:dissipation} that
	\begin{equation}\label{equation:ForLocationOfSpectrum}
	\begin{aligned}
	(1+\tau \omega) \int_{L_1}^{L_2} |q|^2 \, dx & = &&  	-\omega \Bigg\{ \int_{L_1}^{L_2} \left( a |u^1_x|^2 + |\theta|^2 + |u^2|^2 \right) \, dx \\
	& && + \left[ \int_0^{L_1} + \int_{L_2}^{L_3} \right] \left( b |v^1_x|^2  + |v^2|^2 \right) \, dx \Bigg\}
	\end{aligned}
	\end{equation}
	for arbitrary $0 \neq U = [u^1, v^1, \theta, q, u^2, v^2]' \in D(\mathcal{A})$ such that $(\omega + il - \mathcal{A})U = 0$ with $\omega, l \in \mathbb{R}$. Indeed,
	\[
		0 = \Re \langle (\omega + il - \mathcal{A})U, U \rangle_{\mathcal{H}} =  \Re \langle (\omega - \mathcal{A})U, U \rangle_{\mathcal{H}} = \omega \langle U, U \rangle_{\mathcal{H}} + \int_{L_1}^{L_2} q \overline{q} \, dx,
	\]
	which is exactly \eqref{equation:ForLocationOfSpectrum} when the inner product $\langle \cdot, \cdot \rangle_{\mathcal{H}}$ is written out according to its definition. Now assume $\omega < - \tau^{-1}$ would hold, then there would be two cases: 1) The left hand side of \eqref{equation:ForLocationOfSpectrum} is negative but then the right hand side has either the sign of $-\omega$, which is positive due to $- \omega > \tau^{-1} > 0$, or the right hand side is zero; 2) The left hand side is zero, which is only possible if $q = 0$, but because of $U \neq 0$ the right hand side is now strictly positive. Both cases lead to a contradiction and hence no eigenvalue $\omega + il$ of $\mathcal{A}$ can satisfy $\omega < - \tau^{-1}$. In a similar way one could derive that $\omega > 0$ is not possible, but this has already been observed above and follows directly from the Hille \& Yosida theorem. One can conclude that the spectrum of $\mathcal{A}$ is contained in the infinite strip
	\[
	\{ \lambda \in \mathbb{C} \, | \,  -\tau^{-1} \leq \Re \lambda \leq 0 \}.
	\]
\end{remark}

The next remark shall be frequently used in the sequel.
\begin{remark}\label{remark:q_estimate}
	If $U \in D(\mathcal{A})$, $l \in \mathbb{R}$ and $F \in \mathcal{H}$ such that $(il - \mathcal{A})U = F$ in $\mathcal{H}$. Then by \eqref{equation:dissipation} and Cauchy's inequality
	\[
	\int_{L_1}^{L_2} q \overline{q} \, dx = - \Re \langle \mathcal{A} U, U \rangle_{\mathcal{H}} = \Re \langle (il - \mathcal{A)} U, U \rangle_{\mathcal{H}} =  \Re \langle F, U \rangle_{\mathcal{H}} \leq \|U\|_{\mathcal{H}} \|F\|_{\mathcal{H}}.
	\]
\end{remark}

\begin{proposition}\label{proposition:StationarySolutions} One has $0 \in \sigma(\mathcal{A})$ and the stationary solutions to \eqref{equation:Cauchyproblem} are characterized by 
	\[
	\operatorname{ker}(\mathcal{A}) = \operatorname{span}_{\mathcal{H}}\left\{ [\zeta^1, \zeta^2, \zeta^3, 0, 0, 0]' \right\},
	\]
	where
	\begin{equation*}
	\begin{aligned}
	\zeta^1(x) & := && \frac{L_2-L_1-L_3}{L_1 L_3}x+1,\\
	\zeta^2(x) & := && \begin{cases} 
		\frac{L_2-L_1}{L_1L_3}x & x \in [0, L_1], \\
		\frac{L_2-L_1}{L_1L_3}x - \frac{L_2-L_1}{L_1} & x \in [L_2, L_3],
	\end{cases}&& \\
	\zeta^3(x) & := && \frac{a(L_2-L_1-L_3) - b(L_2 -L_1)}{m L_1 L_3}.
	\end{aligned}
	\end{equation*} 
	Moreover, the imaginary axis without the origin is included in $\rho(\mathcal{A})$, hence 
	\[
	i \mathbb{R} \cap \sigma(\mathcal{A}) = \{0\}.
	\]
\end{proposition}
\begin{proof}
	\begin{enumerate}[1)]
		\item Let $U = \alpha [\zeta^1, \zeta^2, \zeta^3 , 0, 0, 0]'$ with $\alpha \in \mathbb{C}\neq \{0\}$ and $\zeta^1, \zeta^2, \zeta^3$ as defined above. Then it holds $\mathcal{A}U = 0$ directly by definition of $\mathcal{A}$ as long as $U \in D(\mathcal{A})$. Hence, it is left to verify that $U \in D(\mathcal{A})$, which follows from
		\begin{equation*}
		\begin{aligned}
		\zeta^1(L_1) & = && \frac{L_2 - L_1 -L_3}{L_3} + \frac{L_3}{L_3} = \frac{L_2 - L_1}{L_3} = \zeta^2(L_1),\\
		\zeta^1(L_2) & = && \frac{L_2 - L_1 -L_3}{L_1L_3}L_2 + \frac{L_1}{L_1} = \frac{L_2 - L_1}{L_1L_3}L_2 - \frac{L_2-L_1}{L_1} = \zeta^2(L_2),\\
		a \zeta^1_x(L_i) - m \zeta^3(L_i) & = && a\frac{L_2-L_1-L_3}{L_1L_3} - \frac{a(L_2 - L_1 -L_3) - b(L_2 - L_1)}{L_1 L_3} \\
		& = && \frac{b(L_2 - L_1)}{L_1 L_3} = b\zeta^2_x(L_i) \quad (i = 1,2),\\
		0 & = && \zeta^2(0) = \zeta^2(L_3). 	   
		\end{aligned} 
		\end{equation*}
		This shows that $0 \in \sigma(\mathcal{A})$ and  $\operatorname{span}_{\mathcal{H}}\left\{[\zeta^1, \zeta^2, \zeta^3 , 0, 0, 0]'\right\} \subseteq \operatorname{ker}(\mathcal{A})$.
		\item Let $U = [u^1, v^1, \theta , q, u^2, v^2]' \in D(\mathcal{A})$ be an eigenvector with respect to $\lambda := 0 \in \sigma(\mathcal{A})$, hence
		\begin{equation*}
		\begin{aligned}
		& u^2 & = && 0 && \mbox{ in } & H^1(L_1, L_2),\\
		& v^2 & = && 0 && \mbox{ in } & H^1((0, L_1) \cup (L_2, L_3)),\\
		& kq_{x} & = && 0 && \mbox{ in } & L^2(L_1, L_2),\\
		&  q + k \theta_{x} & = && 0 && \mbox{ in } & L^2(L_1, L_2), \\
		& - au^1_{xx} + m \theta_{x} & = && 0 && \mbox{ in } & L^2(L_1, L_2), \\
		& - bv^1_{xx} & = && 0 && \mbox{ in } & L^2((0, L_1) \cup (L_2, L_3)).
		\end{aligned}
		\end{equation*}
		We remind that the Poincar\'e inequality argument in \eqref{equation:Poincare} implies that $\langle \cdot, \cdot \rangle_{\mathcal{H}}$ induces the norm $\|\cdot\|_{\mathcal{H}}$ in $\mathcal{H}$ and therefore implicitly justifies that the first two equations hold in $H^1(L_1, L_2)$ and $H^1((0, L_1) \cup (L_2, L_3))$ respectively.
		Inserting $\omega = 0$ into \eqref{equation:ForLocationOfSpectrum} shows that $q = 0$ and furthermore $\theta = \operatorname{const}$ as well as $u^2 = v^2 = 0$. Moreover, one can obtain that $u^1$ and $v^1$ are with constants $\alpha, \beta, c, d ,e, f \in \mathbb{C}$ determined by
		\begin{equation*}
		\begin{aligned}
		u^1(x) & = && \alpha x + \beta \quad (x \in [L_1, L_2]),\\
		v^1(x) & = && (cx + d) \chi_{[0, L_1]}(x) + (ex + f) \chi_{[L_2, L_3]}(x) \quad (x \in [0, L_1] \cup [L_2, L_3]).
		\end{aligned} 
		\end{equation*}
		One can now insert the unknowns $\alpha, \beta, l, c, d, e, f \in \mathbb{C}$ into the boundary and transmission conditions in order to obtain $[u^1, v^1, \theta, q, u^2, v^2]' = \gamma [\zeta^1, \zeta^2, \zeta^3, 0, 0, 0]'$ with $\gamma \in \mathbb{C}$. We omit these calculations but point out that $d = 0$ and, after establishing a triangular form, it is left to solve the algebraic system
		\[
		\left[\begin{array}{cccccc}
		a & 0 & -m & -b & 0 & 0 \\ 
		0 & 1 & 0 & \frac{L_1L_2}{L_1-L_2} & -\frac{L_1L_2}{L_1-L_2} & - \frac{L_1}{L_1-L_2} \\
		0 & 0 & \frac{mL_1}{a} & L_1(\frac{b}{a} - 1) - \frac{L_1L_2}{L_1-L_2} & \frac{L_1L_2}{L_1-L_2} & \frac{L_1}{L_1-L_2} \\ 
		0 & 0 & 0 & 1 & -1 & 0\\
		0 & 0 & 0 & 0 & L_3 & 1 \\ 
		0 & 0 & 0 & 0 & 0 & 0
		\end{array}\right] \VECVECSIX{\alpha}{\beta}{\theta}{c}{e}{f} = \VECVECSIX{0}{0}{0}{0}{0}{0}. 	
		\]
		\item As observed earlier, $\sigma(\mathcal{A})$ consists only of eigenvalues. Assume now that $\lambda = il$ for $l \in \mathbb{R}\setminus\{0\}$ is an eigenvalue of $\mathcal{A}$ with corresponding eigenvector $U = [u^1, v^1, \theta, q, u^2, v^2]' \in D(\mathcal{A})$, which therefore has to obey the equations
		\begin{equation}\label{eq:ResolventEquationsImaginaryAxis}
		\begin{aligned}
		& il u^1 - u^2 & = && 0 && \mbox{ in } & H^1(L_1,L_2),\\
		& il v^1 - v^2 & = && 0 && \mbox{ in } & H^1((0, L_1) \cup (L_2, L_3)),\\
		& il \theta + kq_{x} + ilmu^1_{x} & = && 0 && \mbox{ in } & L^2(L_1,L_2),\\
		& (1 + \tau il) q + k \theta_{x} & = && 0 && \mbox{ in } & L^2(L_1,L_2), \\
		& (il)^2 u^1 - au^1_{xx} + m \theta_{x} & = && 0 && \mbox{ in } & L^2(L_1,L_2), \\
		& (il)^2 v^1 - bv^1_{xx} & = && 0 && \mbox{ in } & L^2((0, L_1) \cup (L_2, L_3)).
		\end{aligned}
		\end{equation}
		Due to $l \neq 0$ one can divide the third and fourth equation in \eqref{eq:ResolventEquationsImaginaryAxis} by $il$ and additionally multiplying the last four equations in \eqref{eq:ResolventEquationsImaginaryAxis} with $\overline{\theta}, \overline{q}, \overline{u^1}, \overline{v^1}$ as well as integrating over $[L_1,L_2]$ and $[0, L_1] \cup [L_2, L_3]$ respectively, results in
		\begin{equation}\label{equation:estResImag}
		\begin{aligned}
		0 & = &&  \|\theta\|^2_{L^2(L_1, L_2)} - l^2 \|u^1\|^2_{L^2(L_1, L_2)} + a \|u^1_x\|^2_{L^2(L_1, L_2)} \\
		& && -l^2 \|v^1\|^2_{L^2((0, L_1) \cup (L_2, L_3))} + b \|v^1_x\|^2_{L^2((0, L_1) \cup (L_2, L_3))}.
		\end{aligned}
		\end{equation}
		Above, $q = 0$ in $L^2(L_1, L_2)$ has been deduced from \eqref{equation:ForLocationOfSpectrum} and since $q$ is continuous by embedding, $q_x$ vanishes as well. One can further infer that $\theta$ is constant due to $\theta_x = 0$, which comes from the fourth equation, and obtain together with $q_x = 0$ that $u^1_x$ equals to a constant in $L^2(L_1, L_2)$. Because of $u^1 \in H^2(L_1, L_2)$, one knows that $u^1_x$ is continuous and therefore everywhere equal to a constant, hence $u^1_{xx} = 0$ in $L^2(L_1, L_2)$. Plugging this into the fifth equation in \eqref{eq:ResolventEquationsImaginaryAxis}, one obtains $u^1 = 0$. Finally, at the interface, $v(L_1) = v(L_2) = 0$ holds due to the respective transmission conditions and as a result the last equation in \eqref{eq:ResolventEquationsImaginaryAxis} can only be satisfied by $v^1 \equiv 0$. Putting $\|u^1\|_{L^2(L_1, L_2)} = \|v^1\|_{L^2((0,L_1) \cup (L_2,L_3))} = 0$ into \eqref{equation:estResImag}, and using the Poincar\'e inequality in the fashion of \eqref{equation:Poincare}, implies $U = 0$ in $\mathcal{H}$, which is a contradiction to the assumption that $\lambda = il$ is an eigenvalue of $\mathcal{A}$.\qed
	\end{enumerate}
\end{proof}

\begin{lemma}\label{lemma:directsum}
	$\mathcal{H} = \operatorname{ker}(\mathcal{A}) \oplus \operatorname{range}(\mathcal{A})$.
\end{lemma}
\begin{proof}
	The equality \eqref{equation:ForLocationOfSpectrum} likewise holds for $\mathcal{A}^*$. Therefore, if $U \in \operatorname{ker}(A)\cup\operatorname{ker}(A^*)$, then $\|q\|_{L^2(L_1, L_2)} = 0$. Hence, by observing the formula for $\mathcal{A}^*$ in the proof of \Cref{proposition:Semigroup} it becomes evident that $\operatorname{ker}(\mathcal{A}) = \operatorname{ker}(\mathcal{A}^*)$. Indeed, $\mathcal{A} + E_{\tau} = - \mathcal{A}^* - E_{\tau}$, where the operator $E_{\tau}\colon \mathcal{H} \to \mathcal{H}$ was defined above in \eqref{equation:adjointFormula}. It is known that, $\mathcal{A}$ being densely defined, $\operatorname{ker}(\mathcal{A}^*) = \operatorname{range}(\mathcal{A})^{\perp}$. Inserting $\operatorname{ker}(\mathcal{A}) = \operatorname{ker}(\mathcal{A}^*)$ concludes the proof. \qed
\end{proof}

\section{Uniform exponential stability}\label{Section:exponentialStability}
In this section we prove the uniform exponential stability for \eqref{equation:TimeDomainSystem:eq1} - \eqref{equation:TimeDomainSystem:eq4}. We adopt hereto the approach from \cite{LiuZheng} and recall the following theorem from the theory of operator semigroups, see Huang \cite{Huang} and Prüss \cite{Pruess}, using here the same formulation as in \cite[Theorem 1.3.1]{LiuZheng}.
\begin{theorem}\label{theorem:quotedCharacterisationExpStab}
	Let $S(t) = \operatornamewithlimits{e}^{At}$ be a strongly continuous semigroup of contractions on a Hilbert space $H$. Then $S(t)$ is exponentially stable if and only if
	\begin{equation}\label{equation:desired}
	\begin{aligned}
	i \mathbb{R} := \{ i l \, | \, l \in \mathbb{R} \} \subseteq \rho(A)
	\end{aligned} 
	\end{equation}
	and
	\begin{equation}\label{equation:desiredResBound}
	\begin{aligned}
	\limsup\limits_{|l| \to +\infty} \|(il - A)^{-1} \|_{L(H)} < +\infty
	\end{aligned} 
	\end{equation}
	hold.
\end{theorem}

In order to deal with the non-trivial kernel of $\mathcal{A}$, we introduce the operator $\mathcal{A}_0$, which is the restriction of $\mathcal{A}$ to $\operatornamewithlimits{range}(\mathcal{A})$ with domain $D(\mathcal{A}_0) = D(\mathcal{A}) \cap \operatornamewithlimits{range}(\mathcal{A})$. By \Cref{lemma:directsum} one has $\sigma(\mathcal{A}_0) = \sigma(\mathcal{A})\setminus\{0\}$, which can be seen in the following way: First take $\lambda \in \sigma(\mathcal{A}_0)$. Because of \Cref{lemma:directsum}, namely $\mathcal{H} = \operatorname{ker}(\mathcal{A}) \oplus \operatorname{range}(\mathcal{A})$, and the definition of $\mathcal{A}_0$, it is $\lambda \neq 0$. Moreover eigenvectors of the restriction $\mathcal{A}_0$ are automatically eigenvectors of $\mathcal{A}$ and overall one has $\lambda \in \sigma(\mathcal{A})\setminus\{0\}$. For the other inclusion take $\lambda \in \sigma(\mathcal{A})\setminus\{0\}$ and arbitrarily fix an eigenvector $U \in D(\mathcal{A})$ corresponding to $\lambda$. One can uniquely decompose $U = U_k + U_i$ where $U_k \in \operatorname{ker}(\mathcal{A})$ and $U_i \in \operatorname{range}(\mathcal{A})$ and as a result
\[
	0 = (\lambda - \mathcal{A})U = (\lambda - \mathcal{A})(U_k + U_i) = \lambda (U_k + U_i) - \mathcal{A} U_i,
\]
hence $U =  \mathcal{A}(\lambda^{-1}U_i)$, which implies $U \in \operatorname{range}(\mathcal{A})$. Therefore, $\lambda \in \sigma(\mathcal{A}_0)$ holds.

The goal for now is to prove that $\mathcal{A}_0$ satisfies \eqref{equation:desired} and \eqref{equation:desiredResBound}. From \Cref{proposition:StationarySolutions} it is already known that $i \mathbb{R} \subseteq \rho(\mathcal{A}_0)$. In order to verify \eqref{equation:desiredResBound}, we argue via contradiction. Hereto, we assume now that \eqref{equation:desiredResBound} would not hold. Then, by taking the negation of \eqref{equation:desiredResBound}, there would exist a sequence $(l_n)_{n \in \mathbb{N}} \subseteq \mathbb{R} \setminus \{0\}$ with $|l_n| \to +\infty$, as $n \to +\infty$, and a sequence $(U_n)_{n \in \mathbb{N}} \subseteq D(\mathcal{A}_0)$ with $\|U_n\|_{\mathcal{H}} = 1$ for all $n \in \mathbb{N}$ such that
\[
\| (i l_n - \mathcal{A}_0) U_n \|_{\mathcal{H}} \to 0, \quad n \to +\infty.
\]
Then the plan is to conclude that this would imply $(U_n)_{n \in \mathbb{N}} \to 0$ in $\mathcal{H}$ as  $n \to \infty$, which would contradict to $\|U_n\|_{\mathcal{H}} = 1$ for all $n \in \mathbb{N}$ and prove that actually \eqref{equation:desiredResBound} must be true. We shall carry out this plan now.

For simplicity of the notation, the index $n \in \mathbb{N}$ is dropped and $l$ is written instead of $l_n$ as well as $U_l$ instead of $U_n$. Moreover we denote the components $U_l = [u^1_l, v^1_l, \theta_l , q_l, u^2_l, v^2_l]'$ as usual and set $F_l = [f^1_l, g^1_l, h_l, p_l, f^2_l, g^2_l]' := (il - \mathcal{A}_0)U_l$, where we also omit the index $n$. Derivatives are denoted with a separating comma, e.g. $\partial_x u^2_l = u^2_{l,\, x}$.  Thus, for $|l| \to +\infty$,
\begin{subequations}\label{equation:RessolventDecaySystem}
	\begin{align}
	& il u^1_l - u^2_l & = && f^1_l &&  \longrightarrow 0 &\quad \mbox{ in } & H^1(L_1,L_2),\label{eq:ResolventEqA:1}\\
	& il v^1_l - v^2_l & = && g^1_l && \longrightarrow 0 & \quad \mbox{ in } & H^1((0,L_1) \cup (L_2,L_3)),\label{eq:ResolventEqA:2}\\
	& il \theta_l + kq_{l,\, x} + mu^2_{l,\, x} & = && h_l && \longrightarrow 0 &\quad \mbox{ in }&  L^2(L_1,L_2),\label{eq:ResolventEqA:3}\\
	& (1 + \tau il) q_l + k \theta_{l,\, x} & = && \tau p_l && \longrightarrow 0 &\quad \mbox{ in }  &L^2(L_1,L_2),\label{eq:ResolventEqA:4} \\
	& il u^2_l - au^1_{l,\, xx} + m \theta_{l,\, x} & = && f^2_l && \longrightarrow 0 & \quad\mbox{ in }&  L^2(L_1,L_2),\label{eq:ResolventEqA:5} \\
	& il v^2_l - bv^1_{l,\, xx} & = && g^2_l && \longrightarrow 0 & \quad\mbox{ in } & L^2((0,L_1) \cup (L_2,L_3))\label{eq:ResolventEqA:6}.
	\end{align}
\end{subequations}
Starting from \eqref{eq:ResolventEqA:1} and \eqref{eq:ResolventEqA:2} one can use the assumption $\|U_l\|_{\mathcal{H}} = 1$ in order to deduce
\begin{equation*}
\begin{aligned}
\|u^1_l\|_{L^2(L_1, L_2)} \to 0, \quad  \|v^1_l\|_{L^2((0,L_1) \cup (L_2,L_3))} \to 0, \quad |l| \to +\infty.
\end{aligned} 
\end{equation*}
Indeed, \eqref{eq:ResolventEqA:1} implies $\|u^1_l\|_{L^2(L_1, L_2)} = |l|^{-1}\|f^1_l + u^2_l\|_{L^2(L_1, L_2)} \to 0$ as $|l| \to + \infty$, because by assumption it holds that $\|f^1_l\|_{L^2(L_1, L_2)} \to 0$ as $|l| \to + \infty$ and $\|u^2_l\|_{L^2(L_1, L_2)} \leq \|U_l\|_{\mathcal{H}} = 1$. For $v^1_l$ the argument works analogously.
Next, from \eqref{eq:ResolventEqA:3} and \eqref{eq:ResolventEqA:5}, by multiplying with $H^1_0(L_1,L_2)$ test functions, one has
\begin{equation*}
\begin{aligned}
|l| \|\theta_l\|_{H^{-1}(L_1, L_2)} & \leq && C (\|q_l\|_{L^2(L_1, L_2)} + \|u^2_l\|_{L^2(L_1, L_2)} + \|F_l\|_{\mathcal{H}} ),\\
|l| \|u^2_l\|_{H^{-1}(L_1, L_2)} & \leq && C (\|u^1_{l,\, x}\|_{L^2(L_1, L_2)} + \|\theta_l\|_{L^2(L_1, L_2)} + \|F_l\|_{\mathcal{H}} ),
\end{aligned} 
\end{equation*}
and interpolation, see \cite[Lemma 12.1.]{LionsMagenes1}, combined with \eqref{equation:RessolventDecaySystem} implies
\begin{equation*}
\begin{aligned}
\|\theta_l\|_{L^2(L_1, L_2)}^2 & \leq && \|\theta_l\|_{H^{-1}(L_1, L_2)}\|\theta_l\|_{H^{1}(L_1, L_2)} \\
& \leq && \frac{1}{|l|} C (\|q_l\|_{L^2(L_1, L_2)} + \|u^2_l\|_{L^2(L_1, L_2)} + \|F_l\|_{\mathcal{H}}) \\
& && \times (\|\theta_l\|_{L^2(L_1, L_2)} + \|\theta_{l,\, x} \|_{L^2(L_1, L_2)}) \\
& = && \frac{1}{|l|} C (\|q_l\|_{L^2(L_1, L_2)} + \|u^2_l\|_{L^2(L_1, L_2)} + \|F_l\|_{\mathcal{H}} ) \\
& && \times (\|\theta_l\|_{L^2(L_1, L_2)} + \| k^{-1}(1 + \tau il) q_l - k^{-1}\tau p_l \|_{L^2(L_1, L_2)}),
\end{aligned} 
\end{equation*}
which gives the convergence
\begin{equation}\label{equation:thetabound}
\begin{aligned}
\|\theta_l\|_{L^2(L_1, L_2)}^2 & \leq &&  C (\|q_l\|_{L^2(L_1, L_2)} + \|u^2_l\|_{L^2(L_1, L_2)} + \|F_l\|_{\mathcal{H}} ) \\
& && \times \left(\frac{1}{|l|} \|\theta_l\|_{L^2(L_1, L_2)} + \left(\frac{\tau}{k} + \frac{1}{k|l|} \right)\| q_l \|_{L^2(L_1, L_2)} + \frac{\tau}{k|l|}\|p_l \|_{L^2(L_1, L_2)}\right)\\
& \to && 0, \quad |l| \to +\infty.
\end{aligned} 
\end{equation}
Indeed, $\|U_l\|_{\mathcal{H}} = 1$ implies that $\|u^2_l\|_{L^2(L_1, L_2)}$, $\|\theta_l\|_{L^2(L_1, L_2)}$, $\|q_l\|_{L^2(L_1, L_2)}$ are bounded independently of $l \in \{l_n \, | \, n \in \mathbb{N} \}$,  $\|F_l\|_{\mathcal{H}} \to 0$ as $|l| \to + \infty$ gives $\|p_l\|_{L^2(L_1, L_2)}, \|f_l\|_{L^2(L_1, L_2)} \to 0$ as $|l| \to + \infty$ and from \Cref{remark:q_estimate} it follows that $\|q_l\|_{L^2(L_1, L_2)} \to 0$ as $|l| \to +\infty$, which explains the convergence for all terms above. Similarly, one has
\begin{equation}\label{equation:u2bound}
\begin{aligned}
\|u^2_l\|_{L^2(L_1, L_2)}^2 & \leq && \|u^2_l\|_{H^{-1}(L_1, L_2)}\|u^2_l\|_{H^{1}(L_1, L_2)} \\
& \leq && \frac{C}{|l|} (\|u^1_{l,\, x}\|_{L^2(L_1, L_2)} + \|\theta_l\|_{L^2(L_1, L_2)} + \|F_l\|_{\mathcal{H}}) \|u^2_l\|_{H^1(L_1, L_2)}\\
& = &&  \frac{C}{|l|} (\|u^1_{l,\, x}\|_{L^2(L_1, L_2)} + \|\theta_l\|_{L^2(L_1, L_2)} + \|F_l\|_{\mathcal{H}})  \| il u^1_l - f_l \|_{H^1(L_1, L_2)},
\end{aligned} 
\end{equation}
leading to the following result, for which a proof shall be provided in \Cref{proofs}.
\begin{lemma}\label{lemma:auxilliary0}
	We have $\|u^1_{l,\, x}\|_{L^2(L_1,L_2)} \to 0$ and $\|u^2_l\|_{L^2(L_1,L_2)} \to 0$ as $|l| \to +\infty$.
\end{lemma}

Therefore it is left to control the undamped, purely elastic, parts of the bar by means of the previous estimates with the help of the transmission conditions at the interface points. 
Hereto we multiply \eqref{eq:ResolventEqA:5} by
\[
w_1(x) := (L_1 + L_2 - 2x)\overline{(au^1_{l,\, x}(x)-m\theta_l(x))},
\]
integrate over $[L_1, L_2]$ and get
\begin{equation}\label{equation:Integral1}
\begin{aligned}
\int_{L_1}^{L_2} \left[(il u^2_l - au^1_{l,\, xx} + m \theta_{l,\, x})w_1 - f^2_l w_1 \right] \, dx = 0.
\end{aligned} 
\end{equation}
By inserting \eqref{equation:RessolventDecaySystem} and performing multiple integrations by part, the following representation for \eqref{equation:Integral1} shall be proved in \Cref{proofs}.

\begin{lemma}\label{lemma:auxilliaryA}
	The equality \eqref{equation:Integral1} is equivalent to
	\begin{equation*}
	\begin{aligned}
	I_1 & = && \frac{L_2-L_1}{2}B_1 + R_1,
	\end{aligned} 
	\end{equation*}
	where
	\begin{equation*}
	\begin{aligned}
	I_1 & := && \int_{L_1}^{L_2} \left[ (a+m^2)|u^2_l|^2 + |a u^1_{l,\, x} - m\theta_l|^2 + \frac{k^2}{\tau} |\theta_l|^2 \right] \, dx, \\
	B_1 & := && \sum\limits_{i=1,2} \left[ (a+m^2)|v^2_l(L_i)|^2 + |bv^1_{l,\, x}(L_i)|^2 + \frac{k^2}{\tau}|\theta_l(L_i)|^2\right],
	\end{aligned} 
	\end{equation*}
	while $R_1$ contains all residual terms and satisfies $R_1 \to 0$, as $|l| \to +\infty$.
\end{lemma}

Regarding the purely elastic parts of the bar we define
\begin{equation*}
\begin{aligned}
w_2(x) := x \overline{v^1_{l,\, x}(x)}, \quad w_3(x) := (x - L_3) \overline{v^1_{l,\, x}(x)},
\end{aligned} 
\end{equation*}
and multiply \eqref{eq:ResolventEqA:6} on $[0, L_1]$ by $w_2$ and on $[L_2, L_3]$ by $w_3$. Straight forward calculations, provided in \Cref{proofs}, lead to the following lemma.
\begin{lemma}\label{lemma:auxilliaryB}
	We have
	\begin{equation*}
	\begin{aligned}
	\frac{1}{2} I_2 & = && \frac{1}{2}B_2 + R_2,
	\end{aligned} 
	\end{equation*}
	where
	\begin{equation*}
	\begin{aligned}
	I_2 & := && \left[ \int_{0}^{L_1} + \int_{L_2}^{L_3} \right] \left( |v^2_l|^2 + b |v^1_{l,\, x}|^2 \right) \, dx, \\
	B_2 & := && L_1 |v^2_l(L_1)|^2 +  b L_1 |v^1_{l,\, x}(L_1)|^2 + (L_3-L_2) |v^2_l(L_2)|^2 +  b (L_3-L_2)|v^1_{l,\, x}(L_2)|^2,
	\end{aligned} 
	\end{equation*}
	while $R_2$ contains again all residual terms and satisfies $R_2 \to 0$, as $|l| \to +\infty$.
\end{lemma}	

Now, one can choose a constant $\tilde{C} >  0$, which is independent of $l \in \{l_n \, | \, n \in \mathbb{N} \}$, such that $B_2 < \tilde{C} B_1$ and combine \Cref{lemma:auxilliaryA} and \Cref{lemma:auxilliaryB} in order to find a $l$-independent constant $C > 0$ such that
\begin{equation*}
\begin{aligned}
I_2 \leq C I_1 + R_3,
\end{aligned} 
\end{equation*}
where $R_3$ combines $R_1$ and $R_2$ and satisfies $R_3 \to 0$, as $|l| \to +\infty$. Moreover, $I_1 \to 0$, as $|l| \to +\infty$ by \eqref{equation:thetabound} and \Cref{lemma:auxilliary0}. All terms which are contained in $\|U_l\|_{\mathcal{H}}$ appear in $I_1$, or $I_2$, or are already seen to converge to zero as $|l| \to + \infty$ in \Cref{lemma:auxilliary0}, by \Cref{remark:q_estimate} or in the discussion after \eqref{equation:RessolventDecaySystem}. Therefore, we arrive at
\[
\|U_l\|_{\mathcal{H}} \to 0, \quad |l| \to +\infty,
\]
which is in contradiction with $\|U_l\|_{\mathcal{H}} = 1$ for all $l \in \{l_n \, | \, n \in \mathbb{N} \}$. We are now in the position to apply \Cref{theorem:quotedCharacterisationExpStab} in order to proof the main result.

\begin{theorem}\label{theorem:MainResult}
	Let $U_0 \in \mathcal{H}$, according to \Cref{lemma:directsum} being uniquely decomposed as $U_0 = W_0 + V_0$, with $W_0 \in \operatorname{ker}(\mathcal{A})$ and $V_0 \in \operatorname{range}(\mathcal{A})$. There exist constants $M > 0$ and $\omega > 0$, which are independent of $U_0$, such that
	\[
	\| T(t) U_0 - W_0\|_{\mathcal{H}} \leq M \operatorname{e}^{-\omega t} \| U_0 \|_{\mathcal{H}} \quad (t \geq 0).
	\]
\end{theorem}

\begin{proof}
	Due to the previous elaborations, \Cref{theorem:quotedCharacterisationExpStab} can be applied and implies uniform exponential stability for the semigroup $(T_0(t))_{t \geq 0}$ generated by $\mathcal{A}_0$. We note that $(T_0(t))_{t \geq 0}$ is the restriction of $(T(t))_{t \geq 0}$ to $\operatornamewithlimits{range}(\mathcal{A})$. Indeed, $\operatornamewithlimits{range}(\mathcal{A})$ is an invariant subspace of $(T(t))_{t \geq 0}$ and one can see that $\mathcal{A}_0$, as defined in the beginning of this section, is the part of $\mathcal{A}$ in $\operatornamewithlimits{range}(\mathcal{A})$ according to the unnumbered Definition from \cite[Page 60]{EngelNagel}. The unnumbered proposition from \cite[Page 60]{EngelNagel}, noting also \cite[Page 43, 5.12]{EngelNagel}, can be applied in order to see that $\mathcal{A}_0$ is the generator of the restriction of $(T(t))_{t \geq 0}$ to $\operatornamewithlimits{range}(\mathcal{A})$. The assertion of \Cref{theorem:MainResult} follows now by noticing that
	\begin{equation*}
	\begin{aligned}
		\| T(t) U_0 - W_0 \|_{\mathcal{H}} & \leq && \| T(t) W_0 - W_0 \|_{\mathcal{H}} + \| T(t) V_0 \|_{\mathcal{H}} \\
		& = && \| T_0(t) V_0 \|_{\mathcal{H}} \\
		& \leq && M \operatorname{e}^{-\omega t} \|V_0\|_{\mathcal{H}} \\
		& \leq &&  M \operatorname{e}^{-\omega t} \|U_0\|_{\mathcal{H}},
	\end{aligned} 
	\end{equation*}
	for all $t \geq 0$. We used above that $W_0 \in \operatorname{ker}(\mathcal{A})$ and the property of a strongly continuous semigroup and its generator to commute, hence
	\[
	\left\|\frac{d}{dt} T(t) W_0 \right\|_{\mathcal{H}} = \|\mathcal{A} T(t) W_0 \|_{\mathcal{H}} = \| T(t) \mathcal{A} W_0 \|_{\mathcal{H}} = 0 \quad (t \geq 0),
	\]
	which implies that $T(t)W_0 = W_0$ for all $t \geq 0$. Furthermore, \Cref{lemma:directsum} and the Pythagorean identity have been employed in order to see that
	\[
		\|U_0\|_{\mathcal{H}}^2 \geq \|U_0\|_{\mathcal{H}}^2 - \|W_0\|_{\mathcal{H}}^2 = \|V_0\|_{\mathcal{H}}^2.
	\]\qed
\end{proof}

\begin{remark}\label{remark:Energydecay}
	The stability of equilibria in \Cref{theorem:MainResult} can also be reformulated in terms of the energy $E_U(t)$ of every mild or classical solution to \eqref{equation:TimeDomainSystem:eq1} - \eqref{equation:TimeDomainSystem:eq4}:	Let $U = U(t,x) = [u(t,x), v(t,x), q(t,x), \theta(t,x)]'$ be a mild or classical solution to \eqref{equation:TimeDomainSystem:eq1} - \eqref{equation:TimeDomainSystem:eq4} with initial data $U_0 = [u_0,v_0, q_0, \theta_0, u_1, v_1]' \in \mathcal{H}$ of the form $U_0 = W_0 + V_0$, where $W_0 \in \operatorname{ker}(\mathcal{A})$ and $V_0 \in \operatorname{range}(\mathcal{A})$. Then there exist constants $M > 0$ and $\omega > 0$, independent of the initial data, such that $E_{U - W_0}(t) \leq  M \operatorname{e}^{-\omega t}$ for $t \geq 0$.
\end{remark}

In \cite[Corollary 3.9]{WangMeng} the same model has been studied in a semilinear context and, without further explanations, the assumption
\begin{equation}\label{equation:AssumptionWellpreparedData}
\begin{aligned}
\int_{L_1}^{L_2}  \theta_0 \, dx  + m (u^1_0(L_2) - u^1_0(L_1)) = 0,
\end{aligned} 
\end{equation}
has been imposed on the initial data. We would like prove that \eqref{equation:AssumptionWellpreparedData} is indeed necessary for exponential decay to zero and in fact, that imposing \eqref{equation:AssumptionWellpreparedData} precisely selects those initial states from which solutions to \eqref{equation:TimeDomainSystem:eq1} - \eqref{equation:TimeDomainSystem:eq4} decay to zero (the trivial steady state). This follows from \Cref{corollary:CharacterisationStationary} below by noting that stationary solutions correspond to $\operatorname{ker}(\mathcal{A})$ and that $\operatorname{range}(\mathcal{A})$, as seen above, characterizes the trajectories which decay to zero.

\begin{proposition}\label{corollary:CharacterisationStationary} Let $U_0 = [u^1_0, v^1_0, \theta_0, q_0, u^2_0, v^2_0]' \in \mathcal{H}$ be arbitrary. Then
	\begin{equation*}
	\begin{aligned}
	\int_{L_1}^{L_2}  \theta_0 \, dx  + m (u^1_0(L_2) - u^1_0(L_1)) = 0 \iff U_0 \in \operatorname{range}(\mathcal{A}).
	\end{aligned} 
	\end{equation*}
\end{proposition}
\begin{proof}
	The assertion follows from \Cref{proposition:StationarySolutions} and \Cref{lemma:directsum} by straight forward calculations of which we shall indicate the main steps. First let
	\[
	U_0 := \mathcal{A} \tilde{U} = [\tilde{u}^2, \tilde{v}^2, -(k\tilde{q} + m \tilde{u}^2_0)_x, -k\tau^{-1}\tilde{\theta}_x - \tau^{-1}\tilde{q}, (a\tilde{u}^1_x - m \tilde{\theta})_x, b\tilde{v}^1_{xx}]',
	\]
	where $\tilde{U} \in D(\mathcal{A})$, i.e. $U_0 \in \operatorname{range}(\mathcal{A})$. Then, using the boundary condition for $\tilde{q}$, one gets
	\begin{equation*}
	\begin{aligned}
	\int_{L_1}^{L_2}  \theta_0 \, dx  + m (u^1_0(L_2) - u^1_0(L_1)) & = && - \int_{L_1}^{L_2} \left( k \tilde{q}_x + m\tilde{u}^2_x \right) \, dx  + m (\tilde{u}^2_0(L_2) - \tilde{u}^2_0(L_1))\\
	& = && - m (\tilde{u}^2_0(L_2) - \tilde{u}^2_0(L_1)) +  (\tilde{u}^2_0(L_2) - \tilde{u}^2_0(L_1))\\
	& = && 0.
	\end{aligned} 
	\end{equation*}
	On the contrary, if \eqref{equation:AssumptionWellpreparedData} is satisfied we argue by contradiction. Thereto, assume that $U_0 = W_0 + V_0$ with $0 \neq W_0 \in \operatorname{ker}(\mathcal{A})$ and $V_0 \in \operatorname{range}(\mathcal{A})$. By plugging the concrete formula for $W_0$ from \Cref{proposition:StationarySolutions} into \eqref{equation:AssumptionWellpreparedData} one obtains after some computations
	\begin{equation*}
	\begin{aligned}
	0 & = && \int_{L_1}^{L_2} \frac{a(L_2-L_1-L_3) - b(L_2 - L_1)}{m L_1 L_3} \, dx + m \left( \frac{L_2-L_1-L_3}{L_1L_3}L_2 - \frac{L_2 - L_1 - L_3}{L_3} \right)\\
	& = && (a + m^2) (L_2-L_1-L_3)(L_2-L_1) - b(L_2-L_1)^2,
	\end{aligned} 
	\end{equation*}	
	hence $0 < b(L_2 - L_1) = (a + m^2) (L_2 - L_1 - L_3) < 0$, which is a contradiction.\qed
\end{proof}

\section{Proofs for \Cref{lemma:auxilliary0}, \Cref{lemma:auxilliaryA} and \Cref{lemma:auxilliaryB}}\label{proofs}
\subsection{Proof for \Cref{lemma:auxilliary0}}
We multiply \eqref{eq:ResolventEqA:3} by $\overline{u^1_{l,\, x}}$, integrate over $[L_1, L_2]$, use \eqref{eq:ResolventEqA:1}, \eqref{eq:ResolventEqA:4}, \eqref{eq:ResolventEqA:5} and employ the boundary condition for $q_l$, in order to obtain
\begin{equation}\label{eq:aux0_first}
\begin{aligned}
\int_{L_1}^{L_2}  \theta_l \overline{u^1_{l,\, x}} \, dx & = && - \int_{L_1}^{L_2} \frac{1}{il} \left[ k q_{l,\, x} \overline{u^1_{l,\, x}} + m u^2_{l,\, x} \overline{u^1_{l,\, x}} - h_l \overline{u^1_{l,\, x}}  \right] \, dx\\
& = && - \int_{L_1}^{L_2} \frac{1}{il} \left[- k q_{l} \overline{u^1_{l,\, xx}} + m (il u^1_{l,\, x} - f^1_{l,\, x})\overline{u^1_{l,\, x}} - h_l \overline{u^1_{l,\, x}}  \right] \, dx\\
& = && - \int_{L_1}^{L_2} \left[ \frac{k}{il} \frac{1}{a}(il q_l \overline{u^2_l} - m q_l \overline{\theta_{l,\, x}} + q_l\overline{f^2_l}) + m |u^1_{l,\, x}|^2 - \frac{1}{il}(h_l + m f^1_{l,\, x})\overline{u^1_{l,\, x}} \right] \, dx\\
& = && 	- \int_{L_1}^{L_2} \left[  \frac{k}{il} \frac{1}{a}\left(il q_l \overline{u^2_l} - \frac{m}{k} q_l (\overline{p_l} -(1- \tau il)\overline{q_l}) + q_l\overline{f^2_l}\right) \right. \\
& && \left. + \, m |u^1_{l,\, x}|^2 - \frac{1}{il}(h_l + m f^1_{l,\, x})\overline{u^1_{l,\, x}}  \right] \, dx.\\
\end{aligned} 
\end{equation}
Moreover, for
\[
R := - \int_{L_1}^{L_2} \left[ \frac{k}{il} \frac{1}{a}\left(il q_l \overline{u^2_l} - \frac{m}{k} q_l (\overline{p_l} -(1- \tau il)\overline{q_l}) + q_l\overline{f^2_l}\right) - \frac{1}{il}(h_l + m f^1_{l,\, x})\overline{u^1_{l,\, x}} \right] \, dx
\]
it holds that $R \to 0$ as $|l| \to +\infty$. This is true because of \Cref{remark:q_estimate} and the assumptions $\|F_l\|_{\mathcal{H}} \to 0$, as $|l| \to +\infty$ and $\|U_l\|_{\mathcal{H}} = 1$ for all $l \in \{l_n \, | \, n \in \mathbb{N} \}$, which guarantee in particular that $u^2_l, u^1_{l,\, x}$ are at least bounded in $L^2(L_1, L_2)$ independently of $l \in \{l_n \, | \, n \in \mathbb{N} \}$ and that $q_l, p_l, h_l, f^1_{l,\, x}, f^2_l \to 0$ in $L^2(L_1, L_2)$ as $|l| \to +\infty$. In order to provide more details on the claimed convergence of $R$, we plug this information into the following estimates for each term of $R$ separately and get, with the help of Hölder's inequality, as $|l| \to + \infty$,
\begin{equation*}
\begin{aligned}
	\left|\int_{L_1}^{L_2}  \frac{k}{il} \frac{1}{a} il q_l \overline{u^2_l} \, dx \right| \leq  \frac{k}{a} \underset{\to 0}{\underbrace{\|q_l\|_{L^2(L_1,L_2)}}}\underset{\mbox{\tiny bounded independently of $l$} }{\underbrace{\|u^2_l\|_{L^2(L_1,L_2)}}} \to 0,\\
	\left|\int_{L_1}^{L_2} \frac{k}{il} \frac{1}{a} \frac{m}{k} q_l \overline{p_l} \, dx \right| \leq \frac{1}{l}\frac{m}{a} \underset{\to 0}{\underbrace{\|q_l\|_{L^2(L_1,L_2)}}}\underset{\to 0}{\underbrace{\|p_l\|_{L^2(L_1,L_2)}}} \to 0,\\
	\left|\int_{L_1}^{L_2} \frac{k}{il} \frac{1}{a} \frac{m}{k} (1- \tau il) q_l\overline{q_l} \, dx \right| \leq  \underset{\mbox{\tiny bounded for $l \to + \infty$}}{\underbrace{\left( \frac{m}{a}\frac{1}{l} + \frac{m\tau}{a}\right)}} \underset{\to 0}{\underbrace{\|q_l\|^2_{L^2(L_1,L_2)}}} \to 0, \\
	\left|\int_{L_1}^{L_2} \frac{k}{il} \frac{1}{a}  q_l\overline{f^2_l} \, dx \right| \leq \frac{k}{l} \frac{1}{a} \underset{\to 0}{\underbrace{\|q_l\|_{L^2(L_1,L_2)}}}\underset{\to 0}{\underbrace{\|f^2_l\|_{L^2(L_1,L_2)}}} \to 0, \\
	\left|\int_{L_1}^{L_2} \frac{1}{il}(h_l + m f^1_{l,\, x})\overline{u^1_{l,\, x}} \, dx \right| \leq  \frac{1}{l} \underset{\mbox{\tiny bounded independently of $l$} }{\underbrace{\|u^1_{l,\, x}\|_{L^2(L_1,L_2)}}} \left( \underset{\to 0}{\underbrace{\|h_l\|_{L^2(L_1,L_2)}}} + m\underset{\to 0}{\underbrace{\|f^1_{l,\, x}\|_{L^2(L_1,L_2)}}} \right) \to 0.
\end{aligned} 
\end{equation*}
In this way we obtain $R \to 0$ as $|l| \to +\infty$ as claimed. For similar estimates below some of these details shall be omitted. Combining the result with \eqref{eq:aux0_first}, one arrives at
\begin{equation*}
\begin{aligned}
\int_{L_1}^{L_2} m |u^1_{l,\, x}|^2  \, dx = R - \int_{L_1}^{L_2} \theta_l \overline{u^1_{l,\, x}} \, dx \to 0, \quad |l| \to +\infty.
\end{aligned} 
\end{equation*}
Hence, $\|u^1_{l, x}\|_{L^2(L_1,L_2)} \to 0$ as $|l| \to +\infty$. This is true because $R \to 0$ as $|l| \to +\infty$ as seen above, $\theta_l \to 0$ in $L^2(L_1, L_2)$ as $|l| \to +\infty$ by \eqref{equation:thetabound} and $\|u^1_{l,\, x}\|_{L^2(L_1,L_2)} \leq \frac{1}{\sqrt{a}}$ for all $l \in \{l_n \, | \, n \in \mathbb{N} \}$ due to $\|U_l\|_{\mathcal{H}} = 1$. Now, through \eqref{equation:u2bound}, one has $\|u^2_l\|_{L^2(L_1,L_2)} \to 0$ as $|l| \to +\infty$. \qed

\subsection{Proof for \Cref{lemma:auxilliaryA}}
 We begin with considering the first term of the integral in \eqref{equation:Integral1}, this means we first treat
\begin{equation*}
\begin{aligned}
J_1 :=  \Re \int_{L_1}^{L_2} il (L_1 + L_2 - 2x) u^2_l \overline{(au^1_{l,\, x} - m \theta_l)} \, dx,
\end{aligned} 
\end{equation*}
and obtain by using \eqref{eq:ResolventEqA:1} and \eqref{eq:ResolventEqA:3}, 
\begin{equation*}
\begin{aligned}
J_1 & = && \Re \left\{ \int_{L_1}^{L_2} (L_1 + L_2 - 2x) \left[ a il u^2_l \overline{u^1_{l,\, x}} - m u^2_l(k \overline{q_{l,\, x}} + m \overline{u^2_{l,\, x}} - \overline{h_l}) \right] \, dx \right\} \\
& = && -  (a+m^2) \Re \int_{L_1}^{L_2} (L_1 + L_2 - 2x) u^2_l \overline{u^2_{l,\, x}}  \, dx \\
& && + \Re \int_{L_1}^{L_2} (L_1+L_2-2x) \left[ m u^2_l \overline{h_l} - a u^2_l \overline{f^1_{l,\, x}} - mk u^2_l \overline{q_{l,\, x}} \right]  \, dx.
\end{aligned} 
\end{equation*}
With the help of the prescribed boundary conditions and integration by parts it follows that
\begin{equation*}
\begin{aligned}
J_1 & = && \frac{(L_2-L_1)(a+m^2)}{2} \left[ \sum\limits_{i=1,2} |u^2_l(L_i)|^2 \right] - \int_{L_1}^{L_2} (a+m^2) |u^2_l|^2 \, dx \\
& && + \Re \int_{L_1}^{L_2} \left\{ (L_1 + L_2 - 2x) \left[ mk u^2_{l,\, x} \overline{q_l} + mu^2_l \overline{h_l} - au^2_l \overline{f^1_{l, x}} \right] - 2mku^2_l \overline{q_l} \right\} \, dx.
\end{aligned} 
\end{equation*}
Thus, with \eqref{eq:ResolventEqA:3} it follows that
\begin{equation*}
\begin{aligned}
J_1 & = &&  \frac{(L_2-L_1)(a+m^2)}{2} \left[ \sum\limits_{i=1,2} |u^2_l(L_i)|^2 \right] - \int_{L_1}^{L_2} (a+m^2) |u^2_l|^2 \, dx \\
& && +  \Re \int_{L_1}^{L_2} k (L_1 + L_2 - 2x) \left[ - il \theta_l - k q_{l,\, x} + h_l \right]\overline{q_l} \, dx \\
& && +  \Re \int_{L_1}^{L_2} \left\{ (L_1 + L_2 - 2x) \left[mu^2_l \overline{h_l} - au^2_l \overline{f^1_{l,\, x}} \right] - 2mku^2_l \overline{q_l} \right\}  \, dx. 
\end{aligned} 
\end{equation*}
By utilizing $u^2_l(L_i) = v^2_l(L_i)$, $i = 1,2$, which is due to $U_l \in D(\mathcal{A})$, and inserting 
\[
il \overline{q_l} = \frac{1}{\tau} \overline{q_l} + \frac{k}{\tau} \overline{\theta_{l,\, x}} - \overline{p_l},
\]
which is due to \eqref{eq:ResolventEqA:4}, one arrives at
\begin{equation*}
\begin{aligned}
J_1 & = &&  \frac{(L_2-L_1)(a+m^2)}{2} \left[ \sum\limits_{i=1,2} |u^2_l(L_i)|^2 \right] - \int_{L_1}^{L_2} (a+m^2) |u^2_l|^2 \, dx \\
& && - \Re \int_{L_1}^{L_2} k (L_1 + L_2 - 2x) \left[ \frac{1}{\tau} \overline{q_l} + \frac{k}{\tau} \overline{\theta_{l,\, x}} - \overline{p_l} \right]\theta_l \, dx \\
& && + \Re \int_{L_1}^{L_2} \left\{ (L_1 + L_2 - 2x) \left[mu^2_l \overline{h_l} - au^2_l \overline{f^1_{l,\, x}} - \frac{k^2}{2} \frac{d}{dx} |q_l|^2 + k h_l \overline{q_l} \right] - 2mku^2_l \overline{q_l} \right\} \, dx\\
& = &&  \frac{L_2-L_1}{2} \sum\limits_{i=1,2} \left[ (a+m^2)|v^2_l(L_i)|^2 + \frac{k^2}{\tau}|\theta_l(L_i)|^2 \right] \\
& &&  - \int_{L_1}^{L_2} \left[ (a+m^2) |u^2_l|^2 + \frac{k^2}{\tau} |\theta_l|^2 \right] \, dx + \tilde{R}_1,
\end{aligned} 
\end{equation*}
with
\begin{equation*}
\begin{aligned}
\tilde{R}_1 & := && \Re \int_{L_1}^{L_2} \left\{ (L_1 + L_2 - 2x) \left[mu^2_l \overline{h_l} - au^2_l \overline{f^1_{l,\, x}} - \frac{k^2}{2} \frac{d}{dx} |q_l|^2 + k h_l \overline{q_l} \right] - 2mku^2_l \overline{q_l} \right\} \, dx \\
& && - \Re \int_{L_1}^{L_2} k (L_1 + L_2 - 2x) \left[ \frac{1}{\tau} \overline{q_l} - \overline{p_l} \right]\theta_l \, dx.
\end{aligned} 
\end{equation*}
Because the $\frac{d}{dx} |q_l|^2$- term in the above integral can be treated with integration by parts and the facts that $q_l(L_1) = q_l(L_2) = 0$ as well as $\|q_l\|_{L^2(L_1,L_2)} \to 0$ for $|l| \to + \infty$, one can conclude with the help of $\|U_l\|_{\mathcal{H}} = 1$ for all $l \in \{l_n \, | \, n \in \mathbb{N} \}$ and $\|F_l\|_{\mathcal{H}} \to 0$ as $|l| \to + \infty$ that $\tilde{R}_1 \to 0$ as $|l| \to +\infty$. Next, we consider 
\begin{equation*}
\begin{aligned}
J_2 := - \Re \int_{L_1}^{L_2} (L_1 + L_2 - 2x)\left(a u^1_{l,\, xx} -m\theta_{l,\, x} + f^2_l\right) \overline{\left(au^1_{l,\, x} - m \theta_l\right)} \, dx,
\end{aligned} 
\end{equation*} 
for which integration by parts and employing the transmission conditions gives
\begin{equation*}
\begin{aligned}
J_2 & = && \frac{L_2 - L_1}{2} \sum\limits_{i=1,2} |bv^1_{l,\, x}(L_i)|^2 - \int_{L_1}^{L_2} |au^1_{l,\, x}-m\theta_l|^2 \, dx \, \\
& &&  \underset{:= \tilde{R}_2}{\underbrace{ - \Re \int_{L_1}^{L_2} (L_1 + L_2 - 2x)f^2_l \overline{\left(au^1_{l,\, x} - m \theta_l\right)} \, dx}},  
\end{aligned} 
\end{equation*} 
with $\tilde{R}_2 \to 0$, as $|l| \to +\infty$, similar to previously obtained convergences. Summed up, we obtain \Cref{lemma:auxilliaryA}. \qed
\subsection{Proof for \Cref{lemma:auxilliaryB}}
Multiplying \eqref{eq:ResolventEqA:6} by $x \overline{v^1_{l,\, x}(x)}$ and inserting \eqref{eq:ResolventEqA:2} results in
\begin{equation*}
\begin{aligned}
0 & = && \Re  \int_{0}^{L_1} x \left[ v^2_l \overline{v^1_{l,\, x}}  - bv^1_{l,\, xx} \overline{v^1_{l,\, x}} - g^2_l \overline{v^1_{l,\, x}} \right] \, dx \\
& = &&  - \Re \int_{0}^{L_1} x\left[\frac{1}{2}\left(\frac{d}{dx} |v^2_l|^2 \right) + \frac{b}{2} \left(\frac{d}{dx} |v^1_{l,\, x}|^2 \right) + \left(v^2_l\overline{g^1_{l,\, x}} + g^2_l \overline{v^1_{l,\, x}}\right)\right] \, dx \\
& = && -\frac{L_1}{2} |v^2_l(L_1)|^2 - \frac{bL_1}{2} |v^1_{l,\, x}(L_1)|^2 + \frac{1}{2} \int_0^{L_1} \left( |v^2_l|^2 + b|v^1_{l,\, x}|^2 \right) \, dx + \tilde{R}_3,
\end{aligned} 
\end{equation*}
where 
\[
\tilde{R}_3 := - \Re \int_0^{L_1} x\left(v^2_l\overline{g^1_{l,\, x}} + g^2_l \overline{v^1_{l,\, x}}\right) \, dx \to 0, \quad |l| \to +\infty.
\] 
The calculations for the $[L_2, L_3]$ part, where \eqref{eq:ResolventEqA:6} is multiplied by $(x - L_3) \overline{v^1_{l,\, x}(x)}$, work analogously. \qed

\section*{Acknowledgements}
	We would like to thank the anonymous referee for many constructive remarks, which helped us to improve the readability of this article. This research was partially supported by the National Natural Science Foundation of China (NNSFC) under Grant No. 11631008.

\appendix
%Title of appendix section should contain the word ''Appendix''
\def\appendixname{Appendix }
\section{Uniform exponential stability in the case of Fourier's law}\label{fourier:Section:exponentialStability}
%but theorems and so on should not contain an additional word ''Appendix''. Additionally to the construction here, the file svepjc3.clo has been modified in order to remove an additional whitespace(because the command below already introduces one)
\def\appendixname{}
We demonstrate in this appendix how exponential stability can be obtained, by the same method as above, when the thermoelastic material is modeled in the classical way using Fourier's law for heat conduction. We present all steps but shall avoid repeated proofs and explanations. From now on, the heat flux $q$ is given by $q = -k \theta_x$, see e.g. \cite{RackeHandbook}, and with constants $a, b , m, k > 0$, the system \eqref{equation:TimeDomainSystem:eq1} becomes 
\begin{equation}\label{fourier:equation:TimeDomainSystem:eq1}
	\begin{aligned}
	& u_{tt} - a u_{xx} + m \theta_x  & = && 0 && \mbox{ in } && [L_1, L_2]\times\mathbb{R}_+, \\
	& \theta_t - k^2 \theta_{xx} + mu_{xt} & = && 0 && \mbox{ in } && [L_1, L_2]\times\mathbb{R}_+, \\
	& v_{tt} - b v_{xx} & = && 0 && \mbox{ in } && [0, L_1] \cup [L_2, L_3] \times\mathbb{R}_+.\\
	\end{aligned}
\end{equation}
	The only difference to \eqref{equation:TimeDomainSystem:eq2}-\eqref{equation:TimeDomainSystem:eq4} is that, instead of the boundary condition for $q$, now $\theta_x(L_1, t) = \theta_x(L_2, t) = 0$ is imposed for $t \geq 0$ and that the initial condition for $q$ is dropped.
The state space is given by
\begin{equation*}
	\begin{aligned}
	\mathcal{H}_F  & := &&   \Big\{ [u^1, v^1, \theta , u^2, v^2]'  \in  H^1(L_1, L_2) \times H^1((0, L_1) \cup (L_2, L_3)) \\
	& &&\times L^2(L_1, L_2) \times L^2(L_1, L_2) \times L^2((0,L_1) \cup (L_2,L_3)) \\
	& &&\mbox{ such that } v^1(0) = v^1(L_3) = 0 \mbox{ and }  u^1(L_i) = v^1(L_i) \mbox{ for } i= 1,2   \Big\},
	\end{aligned}
\end{equation*}
equipped, using \eqref{equation:Poincare}, with the norm $\|\cdot\|_{\mathcal{H}_F}$ induced by the inner product
\begin{equation*}
	\begin{aligned}
	\langle  U, \tilde{U} \rangle_{\mathcal{H}_F} & := && \int_{L_1}^{L_2} \left( a u^1_x \overline{\tilde{u^1_x}} + u^2 \overline{\tilde{u^2}} + \theta \overline{\tilde{\theta}} \right)  \, dx  + \left[\int_0^{L_1} + \int_{L_2}^{L_3}\right] \left( b v^1_x \overline{\tilde{v^1}_x}  + v^2 \overline{\tilde{v^2}} \right) \, dx.
	\end{aligned}
\end{equation*}
Let the linear operator $\mathcal{A}_F \colon \mathcal{H}_F \supseteq D(\mathcal{A}_F) \to \mathcal{H}_F$ be defined for $U =  [u^1, v^1, \theta, u^2, v^2]'$ by  
\[
	\mathcal{A}_FU := [u^2, v^2, k^2\theta_{xx} -mu^2_x, au^1_{xx} -m\theta_x, bv^1_{xx}]',
\]
on the domain
\begin{equation*}
	\begin{aligned}
	D(\mathcal{A}_F) & = && \Big\{ U  \in \mathcal{H}_F \cap \big[ H^2(L_1, L_2) \times H^2((0, L_1) \cup (L_2, L_3)) \times H^2(L_1, L_2) \\
	& && \times H^1(L_1, L_2) \times H^1((0, L_1) \cup (L_2, L_3)) \big]  \,\, \Big| \,\,  u^2(L_i) = v^2(L_i), \\
	& && v^2(0) = v^2(L_3) = 0, \, au^1_x(L_i) - m\theta(L_i) = bv^1_x(L_i), \, \theta_x(L_i) = 0, \, i = 1,2   \Big\},
	\end{aligned}
\end{equation*}
such that \eqref{fourier:equation:TimeDomainSystem:eq1} as well as the initial- and boundary conditions transform into the following Cauchy problem in $\mathcal{H}_F$,
\begin{equation}\label{fourier:equation:Cauchyproblem}
	\begin{aligned}
		& \frac{d}{dt}U(t) & = && & \mathcal{A}_F U(t), && t \geq 0, \\
		& U(0) & = && & U_0 \in \mathcal{H}_F. && 
	\end{aligned}
\end{equation}
As in \Cref{proposition:Semigroup}, $\mathcal{A}_F$ is densely defined and using in particular the boundary conditions prescribed via $\mathcal{H}_F$ and $D(\mathcal{A}_F)$ one can verify that $\mathcal{A}_F^*$ is given by $D(\mathcal{A}_F^*) = D(\mathcal{A}_F)$ and 
\[
\mathcal{A}_F^* [u^1, v^1, \theta, u^2, v^2]' = [-u^2, -v^2, k^2\theta_{xx} + mu^2_x, -au^1_{xx} + m\theta_x, -bv^1_{xx}]'.
\]
Repeating these calculations one finds  $(\mathcal{A}_F^*)^* = \mathcal{A}_F$, which implies that $\mathcal{A}_F$ is closed. 
Moreover, for $U \in D(\mathcal{A}_F) = D(\mathcal{A}_F^*)$ it holds that
\begin{equation}\label{fourier:equation:dissipation}
	\begin{aligned}
	\Re \langle \mathcal{A}_FU, U \rangle_{\mathcal{H}_F} = \Re \langle \mathcal{A}_F^*U, U \rangle_{\mathcal{H}_F} = - k^2 \int_{L_1}^{L_2} \theta_x \overline{\theta_x} \, dx \leq 0.
	\end{aligned}
\end{equation}
Hence, $\mathcal{A}_F$ and $\mathcal{A}_F^*$ are dissipative and, by the Lumer \& Phillips Theorem, $\mathcal{A}_F$ generates a strongly continuous contraction semigroup $(T_F(t))_{t \geq 0}$ on $\mathcal{H}_F$. Moreover, $\sigma(\mathcal{A}_F) \subseteq \{\Re \lambda \leq 0\}$ consists only of eigenvalues due to the compact embedding $D(\mathcal{A}_F) \hookrightarrow \mathcal{H}_F$. 

\begin{remark}\label{fourier:remark:q_estimate}
	For  $U \in D(\mathcal{A}_F)$, $l \in \mathbb{R}$ and $F = (il - \mathcal{A}_F)U$, we have like in \Cref{remark:q_estimate} that $k^2\|\theta_x\|_{L^2(L_1,L_2)}^2\leq \|U\|_{\mathcal{H}_F} \|F\|_{\mathcal{H}_F}$.
\end{remark}

The statement of \Cref{proposition:StationarySolutions} and its proof remain true if one replaces $q$ by $-k\theta_x$  at some points or deletes some expressions which do not appear anymore. 
\begin{proposition}\label{fourier:proposition:StationarySolutions} It is $0 \in \sigma(\mathcal{A}_F)$, $i \mathbb{R} \cap \sigma(\mathcal{A}_F) = \{0\}$ and the stationary solutions of \eqref{fourier:equation:Cauchyproblem} are characterized by $\operatorname{ker}(\mathcal{A}_F) = \operatorname{span}_{\mathcal{H}_F}\left\{ [\zeta^1, \zeta^2, \zeta^3, 0, 0]' \right\}$, where $\zeta^1, \zeta^2$ and $\zeta^3$ are the same as defined in \Cref{proposition:StationarySolutions}.
\end{proposition}
	
Because $\operatorname{ker}(\mathcal{A}_F) = \operatorname{ker}(\mathcal{A}_F^*)$, the following result can be proved like \Cref{lemma:directsum}.
\begin{lemma}\label{fourier:lemma:directsum}
		$\mathcal{H}_F = \operatorname{ker}(\mathcal{A}_F) \oplus \operatorname{range}(\mathcal{A}_F)$.
\end{lemma}
	
Since $\operatorname{ker}(\mathcal{A}_F)$ is non-trivial, the operator $\mathcal{A}_{F,0}$ is introduced as the restriction of $\mathcal{A}_F$ to $\operatornamewithlimits{range}(\mathcal{A}_F)$ with domain $D(\mathcal{A}_{F, 0}) = D(\mathcal{A}_F) \cap \operatornamewithlimits{range}(\mathcal{A}_F)$, and by \Cref{fourier:lemma:directsum} it holds that $\sigma(\mathcal{A}_{F, 0}) = \sigma(\mathcal{A}_F)\setminus\{0\}$.  
In order to apply \Cref{theorem:quotedCharacterisationExpStab}, which will lead to exponential stability for $(T_F(t))_{t \geq 0}$ in the same way as demonstrated for $(T(t))_{t \geq 0}$, we argue by contradiction and assume that \eqref{equation:desiredResBound} is not true for $\mathcal{A}_{F, 0}$. Then there exists a sequence $(l_n)_{n \in \mathbb{N}} \subseteq \mathbb{R}\setminus\{0\}$ with $|l_n| \to +\infty$,  as $n \to +\infty$ and, dropping the index $n$, a sequence $U_l \in D(\mathcal{A}_{F, 0})$ with $\|U_l\|_{\mathcal{H}_F} = 1$ for all $l \in \{l_n \, | \, n \in \mathbb{N} \}$ such that in $\mathcal{H}_F$ we have
\[
	F_l = [f^1_l, g^1_l, h_l, p_l, f^2_l, g^2_l]' := (il - \mathcal{A}_{F, 0})U_l \to 0, \quad |l| \to +\infty.
\]
Then, for $|l| \to +\infty$, it holds that
\begin{subequations}\label{fourier:equation:RessolventDecaySystem}
	\begin{align}
		& il u^1_l - u^2_l & = && f^1_l &&  \longrightarrow 0 &\quad \mbox{ in } & H^1(L_1,L_2),\label{fourier:eq:ResolventEqA:1}\\
		& il v^1_l - v^2_l & = && g^1_l && \longrightarrow 0 & \quad \mbox{ in } & H^1((0, L_1) \cup (L_2, L_3)),\label{fourier:eq:ResolventEqA:2}\\
		& il \theta_l - k^2\theta_{l,\, xx} + mu^2_{l,\, x} & = && h_l && \longrightarrow 0 &\quad \mbox{ in }&  L^2(L_1,L_2),\label{fourier:eq:ResolventEqA:3}\\
		& il u^2_l - au^1_{l,\, xx} + m \theta_{l,\, x} & = && f^2_l && \longrightarrow 0 & \quad\mbox{ in }&  L^2(L_1,L_2),\label{fourier:eq:ResolventEqA:4} \\
		& il v^2_l - bv^1_{l,\, xx} & = && g^2_l && \longrightarrow 0 & \quad\mbox{ in } & L^2((0,L_1) \cup (L_2,L_3))\label{fourier:eq:ResolventEqA:5}.
	\end{align}
\end{subequations}
For the model with Cattaneo's law, in the context of \Cref{Section:exponentialStability}, the equation \eqref{eq:ResolventEqA:4} prevented one to conclude $\theta_{l,\, x} \to 0$ in $L^2$ as $|l| \to + \infty$. However, here this convergence holds due to \Cref{fourier:remark:q_estimate}, which leads to simplifications at some points. We continue as in \Cref{Section:exponentialStability} and observe from \eqref{fourier:eq:ResolventEqA:1} and \eqref{fourier:eq:ResolventEqA:2} that $\|u^1_l\|_{L^2(L_1, L_2)} \to 0$, $\|v^1_l\|_{L^2((0,L_1) \cup (L_2,L_3))} \to 0$ as $|l| \to +\infty$, as well as from \eqref{fourier:eq:ResolventEqA:3} and \eqref{fourier:eq:ResolventEqA:4} that
\begin{equation*}
	\begin{aligned}
	|l| \|\theta_l\|_{H^{-1}(L_1, L_2)} & \leq && C (\|\theta_{l,\, x}\|_{L^2(L_1, L_2)} + \|u^2_l\|_{L^2(L_1, L_2)} + \|F_l\|_{\mathcal{H}_F} ),\\
	|l| \|u^2_l\|_{H^{-1}(L_1, L_2)} & \leq && C (\|u^1_{l,\, x}\|_{L^2(L_1, L_2)} + \|\theta_l\|_{L^2(L_1, L_2)} + \|F_l\|_{\mathcal{H}_F} ).
	\end{aligned} 
\end{equation*}
By interpolation and \Cref{fourier:remark:q_estimate}, noting that now $\|\theta_{l,\, x}\|_{L^2(L_1, L_2)} \to 0$ for $|l| \to +\infty$, it holds
\begin{equation}\label{fourier:equation:thetabound}
	\begin{aligned}
	\|\theta_l\|_{L^2(L_1, L_2)}^2 & \leq && \|\theta_l\|_{H^{-1}(L_1, L_2)}\|\theta_l\|_{H^{1}(L_1, L_2)} \\
	& \leq && \frac{1}{|l|} C (\|\theta_{l,\, x}\|_{L^2(L_1, L_2)} + \|u^2_l\|_{L^2(L_1, L_2)} + \|F_l\|_{\mathcal{H}_F}) \\
	& && \times (\|\theta_l\|_{L^2(L_1, L_2)} + \|\theta_{l,\, x} \|_{L^2(L_1, L_2)}) \\
	& \to && 0, \quad |l| \to +\infty.
	\end{aligned} 
\end{equation}
Moreover, one has
\begin{equation}\label{fourier:equation:u2bound}
	\begin{aligned}
	\|u^2_l\|_{L^2(L_1, L_2)}^2 & \leq && \|u^2_l\|_{H^{-1}(L_1, L_2)}\|u^2_l\|_{H^{1}(L_1, L_2)} \\
	& \leq && \frac{C}{|l|} (\|u^1_{l,\, x}\|_{L^2(L_1, L_2)} + \|\theta_l\|_{L^2(L_1, L_2)} + \|F_l\|_{\mathcal{H}_F}) \|u^2_l\|_{H^1(L_1, L_2)}\\
	& = && \frac{C}{|l|} (\|u^1_{l,\, x}\|_{L^2(L_1, L_2)} + \|\theta_l\|_{L^2(L_1, L_2)} + \|F_l\|_{\mathcal{H}_F}) \| il u^1_l - f_l \|_{H^1(L_1, L_2)}.
	\end{aligned} 
\end{equation}
Now, we can conclude similar to \Cref{lemma:auxilliary0} that $u^1_{l,\, x}$ and $u^2_l$ converge to zero. Hereto, by using \eqref{fourier:eq:ResolventEqA:1}, \eqref{fourier:eq:ResolventEqA:3}, \eqref{fourier:eq:ResolventEqA:4}, as well as the boundary condition for $\theta_{l,\, x}$, we obtain
		\begin{equation}\label{fourier:equation:aux_estimate1}
		\begin{aligned}
		0 & = && \int_{L_1}^{L_2} \left[ \theta_l \overline{u^1_{l,\, x}} +  \frac{1}{il} \left( k^2 \theta_{l,\, x} \overline{u^1_{l,\, xx}} + m(ilu^1_{l,\, x} - f^1_{l,\, x})\overline{u^1_{l,\, x}} -h_l \overline{u^1_{l,\, x}} \right) \right] \, dx\\
		& = && \int_{L_1}^{L_2} \left[ \theta_l \overline{u^1_{l,\, x}} + m|u^1_{l,\, x}|^2 + \frac{1}{il} \left( \frac{k^2}{a} \theta_{l,\, x} (-il\overline{u^2_l} + m \overline{\theta_{l,\, x}} - \overline{f^2_l}) - mf^1_{l,\, x}\overline{u^1_{l,\, x}} - h_l \overline{u^1_{l,\, x}} \right)\right] \, dx.\\
		\end{aligned} 
		\end{equation}
		From $\|U_l\|_{\mathcal{H}_F} = 1$, $\|F_l\|_{\mathcal{H}_F} \to 0$ as $|l| \to +\infty$, \eqref{fourier:equation:thetabound} and \Cref{fourier:remark:q_estimate}, one can deduce
		\begin{equation*}
		\begin{aligned}
		\int_{L_1}^{L_2} \left[ \theta_l \overline{u^1_{l,\, x}} + \frac{1}{il} \left( \frac{k^2}{a} \theta_{l,\, x} (-il\overline{u^2_l} + m \overline{\theta_{l,\, x}} - \overline{f^2_l}) - mf^1_{l,\, x}\overline{u^1_{l,\, x}} - h_l \overline{u^1_{l,\, x}} \right) \right] \, dx \to 0, \, |l| \to +\infty,
		\end{aligned} 
		\end{equation*}
which implies together with \eqref{fourier:equation:aux_estimate1} that $\|u^1_{l,\, x}\|_{L^2(L_1,L_2)} \to 0$ as $|l| \to +\infty$ and, by employing \Cref{fourier:equation:u2bound}, one also has $\|u^2_l\|_{L^2(L_1,L_2)} \to 0$ as $|l| \to +\infty$. 

In order to obtain controls for $\|v^1_{l,\, x}\|_{L^2((0, L_1) \cup (L_2, L_3))}$ and $\|v^2_l\|_{L^2((0, L_1) \cup (L_2, L_3))}$, we first multiply \eqref{fourier:eq:ResolventEqA:4} with $\overline{u^1_{l,\, x}}(L_2 - x)$, integrate over $(L_1, L_2)$, take real parts, incorporate \eqref{fourier:eq:ResolventEqA:1} and deduce
\begin{equation}\label{fourier:equation:boundary_estimate1}
	\begin{aligned}
	0 & = && - \int_{L_1}^{L_2} \left[ \frac{1}{2} \frac{d}{dx} |u^2_l|^2  + \frac{a}{2} \frac{d}{dx} |u^1_{l,\, x}|^2 \right](L_2 - x) \, dx  \\
	& && - \underset{=: \tilde{R}_4}{\underbrace{ \Re \int_{L_1}^{L_2} \left[ u^2_l \overline{f^1_{l\, x}} - m \theta_{l,\, x} \overline{u^1_{l,\, x}} + f^2_l\overline{u^1_{l,\, x}} \right] (L_2 - x) \, dx }}.
	\end{aligned} 
\end{equation}
Repeating this calculation, but multiplying \eqref{fourier:eq:ResolventEqA:4} now with $\overline{u^1_{l,\, x}}(L_1 - x)$, results in
\begin{equation}\label{fourier:equation:boundary_estimate2}
	\begin{aligned}
	0 & = && -\int_{L_1}^{L_2} \left[ \frac{1}{2} \frac{d}{dx} |u^2_l|^2 + \frac{a}{2} \frac{d}{dx} |u^1_{l,\, x}|^2 \right] (L_1 - x) \, dx  \\
	& && - \underset{=: \tilde{R}_5}{\underbrace{ \Re \int_{L_1}^{L_2} \left[ u^2_l \overline{f^1_{l\, x}} - m \theta_{l,\, x} \overline{u^1_{l,\, x}} + f^2_l\overline{u^1_{l,\, x}} \right] (L_1 - x) \, dx }},
	\end{aligned} 
\end{equation}
and summing \eqref{fourier:equation:boundary_estimate1} and \eqref{fourier:equation:boundary_estimate2}, as well as using integration by parts, leads to
\begin{equation}\label{fourier:equation:boundary_estimate3}
	\begin{aligned}
	\int_{L_1}^{L_2} \left[ |u^2_l|^2 + a  |u^1_{l,\, x}|^2 \right] \, dx + \tilde{R}_4 + \tilde{R}_5 & = &&  \frac{L_2-L_1}{2} \sum\limits_{i=1,2} \left[ |u^2_l(L_i)|^2 + a|u^1_{l,\, x}(L_i)|^2 \right].
	\end{aligned} 
\end{equation}
Both $\tilde{R}_4$ and $\tilde{R}_5$ converge to zero for $|l| \to +\infty$ by the usual argument. In a similar manner as before, by multiplying \eqref{fourier:eq:ResolventEqA:5} with $\overline{v^1_{l,\, x}}x$ and integrating over $(0, L_1)$, followed by multiplying \eqref{fourier:eq:ResolventEqA:5} with $\overline{v^1_{l,\, x}}(L_3-x)$ and integrating over $(L_2, L_3)$, one obtains
\begin{equation}\label{fourier:equation:boundary_estimate4}
	\begin{aligned}
	0 & = && \frac{b}{2} \|v^1_{l,\, x}\|_{L^2((0, L_1) \cup (L_2, L_3))}^2 + \frac{1}{2}\|v^2\|_{L^2((0, L_1) \cup (L_2, L_3))}^2   \\
	& && - \frac{1}{2} L_1 [|v^2_l(L_1)|^2 + b|v^1_{l,\, x}(L_1)|^2] - \frac{1}{2} (L_3 - L_2) [|v^2_l(L_2)|^2 + b|v^1_{l,\, x}(L_2)|^2]\\
	& && + \underset{=: \tilde{R}_6}{\underbrace{\Re \int_{L_2}^{L_3} \left[v^2_l\overline{g^1_{l,\, x}} + v^2_l\overline{g^1_{l,\, x}} \right](L_3 - x) \, dx - \Re \int_0^{L_1}  \left[ v^2_l\overline{g^1_{l,\, x}} + g^2_l\overline{v^1_{l,\, x}} \right]x \, dx}}.
	\end{aligned} 
\end{equation}
Note that, due to the usual Sobolev embedding, $|\theta_l(x)| \leq C \|\theta_l\|_{H^1(L_1, L_2)}$ for all $x \in [L_1, L_2]$, where $C > 0$ is independent of $l \in \{l_n \, | \, n \in \mathbb{N} \}$. Therefore, there exist $C_1, C_2 > 0$, independent  of $l$, such that for
\begin{equation*}
	\begin{aligned}
	I := \|v^1_{l,\, x}\|_{L^2((0, L_1) \cup (L_2, L_3))}^2 + \|v^2\|_{L^2((0, L_1) \cup (L_2, L_3))}^2 ,
	\end{aligned} 
\end{equation*}
one can infer from \eqref{fourier:equation:boundary_estimate4} and the transmission conditions that
\begin{equation}\label{fourier:equation:boundary_estimate5}
	\begin{aligned}
	I  & \leq && C_1 \sum\limits_{i=1,2} \left[ |u^2_l(L_i)|^2 + |u^1_{l,\, x}(L_i)|^2 + |\theta_l(L_i)|^2 + |u^1_{l,\, x}(L_i)||\theta_l(L_i)| \right] + C_1|\tilde{R}_6| \\
	& \leq && C_2 \sum\limits_{i=1,2} \left[ |u^2_l(L_i)|^2 + |u^1_{l,\, x}(L_i)|^2 + \|\theta_l\|_{H^1(L_1, L_2)}^2 + |u^1_{l,\, x}(L_i)|\|\theta_l\|_{H^1(L_1, L_2)} \right] +  C_1|\tilde{R}_6|,
	\end{aligned} 
\end{equation}
where $\tilde{R}_6$ vanishes as $|l| \to +\infty$. Furthermore, \eqref{fourier:equation:boundary_estimate3} allows to infer $|u^2_l(L_i)| + |u^1_{l,\, x}(L_i)| \to 0$, as $|l| \to +\infty$, for $i=1,2$ and \Cref{fourier:remark:q_estimate} together with \eqref{fourier:equation:thetabound} implies $\| \theta_l \|_{H^1(L_1, L_2)} \to 0$, as $|l| \to +\infty$. As a result, from \eqref{fourier:equation:boundary_estimate5} it follows that $I \to 0$, as $|l| \to +\infty$. In total, all terms contained in $\|U_l\|_{\mathcal{H}_F}$ are seen to converge to zero for large $|l|$ and we arrive at
\[
	\|U_l\|_{\mathcal{H}_F} \to 0, \quad |l| \to +\infty,
\]
which is in contradiction with $\|U_l\|_{\mathcal{H}_F} = 1$. 
	
Above we verified the requirements of \Cref{theorem:quotedCharacterisationExpStab} for $\mathcal{A}_{F,0}$ and therefore conclude the following exponential stability result, which can be proved in the same way as \Cref{theorem:MainResult}.
\begin{theorem}\label{fourier:theorem:MainResult}
		Let $U_0 \in \mathcal{H}_F$, according to \Cref{fourier:lemma:directsum} be uniquely decomposed as $U_0 = W_0 + V_0$, with $W_0 \in \operatorname{ker}(\mathcal{A}_F)$ and $V_0 \in \operatorname{range}(\mathcal{A}_F)$. There exist constants $M > 0$ and $\omega > 0$, which are independent of $U_0$, such that
		\[
		\| T_F(t) U_0 - W_0\|_{\mathcal{H}_F} \leq M \operatorname{e}^{-\omega t} \| U_0 \|_{\mathcal{H}_F} \quad (t \geq 0).
		\]
\end{theorem}

\bibliographystyle{plain}
\bibliography{references}

\end{document}